\documentclass[12pt]{article}
\usepackage[margin =1in]{geometry}
\usepackage{amssymb,amsthm}
\usepackage{amsmath}
\usepackage{enumitem}

\usepackage{hyperref}
\usepackage{cite}
\usepackage{color}

\usepackage{cleveref}

\newcommand{\R}{\mathbb{R}}
\newcommand{\cl}{\textrm{cl}\,}

\newcommand{\Rd}{\mathbb{R}^d}

\newcommand{\inter}{{\rm int}}

\newcommand{\lipsymb}{\mathsf{L}}

\newcommand{\la}{\langle}
\newcommand{\ra}{\rangle}
\newcommand{\ip}[2]{\la {#1}, {#2} \ra}
\newcommand{\Et}[1]{ \mathbb{E}_t \left[ {#1} \right] }
\newcommand{\E}[1]{ \mathbb{E} \left[ {#1} \right] }
\newcommand{\Exi}[1]{ \mathbb{E}_\xi \left[ {#1} \right] }

\newcommand{\prox}{{\rm prox}}

\DeclareMathOperator*{\argmin}{argmin}

\DeclareMathOperator*{\argmax}{argmax}
\newcommand{\ri}{{\rm ri}}
\newcommand{\dom}{{\rm dom}\,}
\newcommand{\EE}{\mathbb{E}}

\numberwithin{equation}{section}

\newcommand{\inclu}[0] {\ar@{^{(}->}}

\newcommand{\RR}{\mathbb{R}}

\newtheorem{lemma}{Lemma}[section]

\newtheorem{theorem}[lemma]{Theorem}
\newtheorem{corollary}[lemma]{Corollary}

\newtheorem{proposition}[lemma]{Proposition}

\newtheorem{example}[lemma]{Example}

\theoremstyle{remark}

\usepackage{mathtools}
\DeclarePairedDelimiter{\dotp}{\langle}{\rangle}
\usepackage[boxruled]{algorithm2e}

\begin{document}
	
	\title{Stochastic model-based minimization \\under high-order growth}
	
	
		\author{
	Damek Davis
		\thanks{School of Operations Research and Information Engineering, Cornell University, Ithaca, NY 14850; 	
		\texttt{people.orie.cornell.edu/dsd95/}.} 
	\and
	Dmitriy Drusvyatskiy
		\thanks{Department of Mathematics, University of Washington, Seattle, WA 98195; 	
		\texttt{sites.math.washington.edu/{\raise.17ex\hbox{$\scriptstyle\sim$}}ddrusv}. Research of Drusvyatskiy was partially supported by the AFOSR YIP award FA9550-15-1-0237 and by the NSF DMS   1651851 and CCF 1740551 awards.}
	\and
	Kellie J. MacPhee
		\thanks{Department of Mathematics, University of Washington, Seattle, WA 98195; 	
		\texttt{sites.math.washington.edu/{\raise.17ex\hbox{$\scriptstyle\sim$}}kmacphee}.} 
	}

	\date{}
	\maketitle

\begin{abstract}
Given a nonsmooth, nonconvex minimization problem, we consider algorithms that iteratively sample and minimize stochastic convex models of the objective function. Assuming that the one-sided approximation quality and the variation of the models is controlled by a Bregman divergence, we show that the scheme drives a natural stationarity measure to zero at the rate $O(k^{-1/4})$. Under additional convexity and relative strong convexity assumptions, the function values converge to the minimum at the rate of $O(k^{-1/2})$ and $\widetilde{O}(k^{-1})$, respectively. We discuss consequences for stochastic proximal point, mirror descent, regularized Gauss-Newton, and saddle point algorithms. 
\end{abstract}

\section{Introduction}
Common stochastic optimization algorithms proceed as follows. Given an  iterate $x_t$, the method samples a model  of the objective function formed at $x_t$ and  declares the next iterate to be a minimizer of the model regularized by a proximal term.
Stochastic proximal point, proximal subgradient, and Gauss-Newton type methods are common examples. 
Let us formalize this viewpoint, following \cite{stochastic_subgrad}. Namely, consider the optimization problem 
\begin{equation}\label{eqn:full_prob_class}
\min_{x\in\R^d}~ F(x):=f(x)+r(x).
\end{equation}
where the function $r\colon\R^d\to\R\cup\{\infty\}$ is closed and convex  and the only access to $f\colon \R^d\to\R$ is by sampling a {\em stochastic one-sided model}. That is, for every point $x$, there exists a family of models $f_x(\cdot,\xi)$ of $f$, indexed by a random variable $\xi\sim P$.  This setup immediately motivates the following algorithm, analyzed in  \cite{stochastic_subgrad}:
\begin{equation}\label{eqn:main_alg}
\left\{~\begin{aligned}
&\textrm{Sample }\xi_t\sim P,\\
&\textrm{Set } x_{t+1}=\argmin_x~ \left\{f_{x_t}(x, \xi_t)+r(x) +\frac{1}{2\eta_t}\|x-x_t\|^2_2\right\}
\end{aligned}~\right\},
\end{equation}
where $\eta_t>0$ is an appropriate control sequence that governs the step-size of the algorithm.

Some thought shows that convergence guarantees of the method \eqref{eqn:main_alg} should rely at least on two factors: $(i)$ control over the approximation quality, $f_{x}(\cdot,\xi)-f(\cdot)$, and $(ii)$ growth/stability properties of the individual models $f_{x}(\cdot,\xi)$. With this in mind, the paper \cite{stochastic_subgrad} isolates the following assumptions:
\begin{equation}\label{eqn:err_approx_intro}
\mathbb{E}_{\xi}[f_{x}(x,\xi)]=f(x)\qquad \textrm{and}\qquad \mathbb{E}_{\xi}[f_{x}(y,\xi)-f(y)]\leq \frac{\tau}{2}\|y-x\|^2_2\qquad\forall x,y,
\end{equation}
and there exists a square integrable function $L(\cdot)$ satisfying
\begin{equation}\label{eqn:variation}
f_{x}(x,\xi)-f_x(y,\xi)\leq L(\xi) \|x-y\|_2\qquad\forall x,y.
\end{equation}
 Condition \eqref{eqn:err_approx_intro} simply says that in expectation, the model $f_{x}(\cdot,\xi)$ must
globally lower bound $f(\cdot)$ up to a quadratic error, while agreeing with $f$ at the base point $x$; when \eqref{eqn:err_approx_intro} holds, the paper \cite{stochastic_subgrad} calls the assignment  $(x,y,\xi)\mapsto f_{x}(y, \xi)$ a stochastic one-sided model of $f$. Property \eqref{eqn:variation}, in contrast, asserts a Lipschitz type property of the individual models $f_x(\cdot,\xi)$.\footnote{The stated assumption (A4) in \cite{stochastic_subgrad} is stronger than \eqref{eqn:variation}; however, a quick look at the arguments shows that property \eqref{eqn:variation} suffices to obtain essentially the same convergence guarantees. } The main result of \cite{stochastic_subgrad} shows that under these assumption, the scheme~\eqref{eqn:main_alg} drives a natural stationarity measure of the problem to zero at the rate $O(k^{-1/4})$. Indeed, the stationarity measure is simply the gradient of the Moreau envelope
\begin{equation}\label{eqn:envel_noneuc_intro}
  F_\lambda (x) := \inf_{y} \left\{F (y) + \tfrac{1}{2\lambda} \|y-x\|^2_2 \right\} ,
\end{equation}
where $\lambda>0$ is a smoothing parameter on the order of $\tau$.

The assumptions~\eqref{eqn:err_approx_intro} and \eqref{eqn:variation} are perfectly aligned with existing literature. Indeed, common first-order algorithms rely on global Lipschitz continuity of the objective function or of its gradient; see for example the monographs \cite{nem_yud,nesterov2013introductory,beck_book}. Recent work  \cite{descentBBT,rel_smooth_freund,Lu_mirror_weird,rick_rel_smooth,nonconv_teb}, in contrast, has emphasized  that global Lipschitz assumptions can easily fail for well-structured problems. Nonetheless, these papers show that it is indeed possible to develop efficient algorithms even without the global Lipschitz assumption. The key idea, originating in \cite{descentBBT,rel_smooth_freund,Lu_mirror_weird}, is to model errors in approximation by a Bregman divergence, instead of a norm. The ability to deal with problems that are not globally Lipschitz is especially important in stochastic nonconvex settings, where line-search strategies that exploit local Lipschitz continuity are not well-developed. 

Motivated by the recent work on relative continuity/smoothness \cite{descentBBT,rel_smooth_freund,Lu_mirror_weird}, 
 we extend the results of \cite{stochastic_subgrad} to non-globally Lipschitzian settings. Formally, we simply  replace the squared norm $\frac{1}{2}\|\cdot\|^2$ in the displayed equations  \eqref{eqn:main_alg}-\eqref{eqn:envel_noneuc_intro} by a Bregman divergence 
\begin{equation*} 
D_\Phi (y,x) = \Phi(y) - \Phi(x) - \la \nabla \Phi (x), y-x \ra,
\end{equation*}
generated by a Legendre function $\Phi$. With this modification and under mild technical conditions, we will show that algorithm~\eqref{eqn:main_alg} drives the gradient of the Bregman envelope~\eqref{eqn:envel_noneuc_intro} to zero at the rate $O(k^{-1/4})$, where the size of the gradient is measured in the local norm induced by $\Phi$.
As a consequence, we obtain new convergence guarantees for stochastic proximal point, mirror descent\footnote{This work appears on arXiv a month after a preprint of Zhang and He \cite{zhang_he}, who provide similar convergence guarantees specifically for the stochastic mirror descent algorithm. The results of the two papers were obtained independently and are complementary to each other.}, and regularized Gauss-Newton methods, as well as for an elementary algorithm for stochastic saddle point problems. Perhaps the most important application arena is when the functional components of the problem grow at a polynomial rate. In this setting, we present a simple Legendre function $\Phi$ that satisfies the necessary assumptions for the convergence guarantees to take hold.  
We also note that the stochastic mirror descent algorithm that we present here does not require mini-batching the gradients, in contrast to the previous seminal work \cite{ghadimilanzang}. 
 
 When the stochastic models $f_{x}(\cdot,\xi)$ are themselves convex and globally under-estimate $f$ in expectation, we prove that the scheme drives the expected functional error to zero at the rate $O(k^{-1/2})$. The rate improves to $\widetilde{O}(k^{-1})$ when the regularizer $r$ is $\mu$-strongly convex relative to $\Phi$ in the sense of \cite{rel_smooth_freund}. In the special case of mirror descent, these guarantees extend the results for convex unconstrained problems in \cite{Lu_mirror_weird} to the proximal setting. Even specializing to the proximal subgradient method, the convergence guarantees appear to be different from those available in the literature. Namely, previous complexity estimates \cite{BelloCruz2017,duchi2009efficient} depend on the largest norms of the subgradients of $r$ along the iterate sequence, whereas Theorems~\ref{thm:convergenceC_nonstrong} and \ref{thm:convergenceC} replace this dependence only by the initial error $r(x_0)-\inf r$.
 
The outline of the manuscript is as follows. Section~\ref{sec:notation} reviews the relevant concepts of convex analysis, focusing on Legendre functions and the Bregman divergence. Section~\ref{sec:problem_class_alg} introduces the problem class and the algorithmic framework. This section also interprets the assumptions made for the stochastic proximal point, mirror descent, and regularized Gauss-Newton methods, as well as for a stochastic approximation algorithm for  saddle point problems. Section~\ref{sec:conv_guarant} discusses the stationarity measure we use to quantify the rate of convergence. Section~\ref{sec:conv_anal} contains the complete convergence analysis of the stochastic model-based algorithm. Section~\ref{sec:spec_anal_mirror descent} presents a specialized analysis for the mirror descent algorithm when $f$ is smooth and the stochastic gradient oracle has finite variance. Finally, in Section~\ref{sec:convexity} we prove convergence rates 
in terms of function values for stochastic model-based algorithms under (relative strong) convexity assumptions.

\section{Legendre functions and the Bregman divergence}\label{sec:notation}
Throughout, we follow standard notation from convex analysis, as set out for example by Rockafellar \cite{rockafellar}.  
The symbol $\R^d$ will denote an Euclidean space with inner product $\langle\cdot,\cdot \rangle$ and the induced norm $\|x\|_2=\sqrt{\langle x, x\rangle}$. For any set $Q\subset\R^d$, we let $\inter\, Q$ and $\cl Q$  denote the interior and closure of $Q$, respectively. Whenever $Q$ is convex, the set $\ri\, Q$ is the interior of $Q$ relative to its affine hull. The effective domain of any function $f\colon\R^d\to\R\cup\{\infty\}$, denoted by $\dom f$,  consists of all points where $f$ is finite. Abusing notation slightly, we will use the symbol $\dom(\nabla f)$ to denote the set of all points where $f$ is differentiable.

This work analyzes stochastic model-based minimization algorithms, where the ``errors''  are controlled by a Bregman divergence. For wider uses of the Bregman divergence in first-order methods, we refer the interested reader to the expository articles of Bubeck~\cite{bubeck2015convex},  Juditsky-Nemirovski \cite{juditsky_NEMfirs}, and Teboulle \cite{Teboulle2018}. 

Henceforth, we fix a
{\em Legendre function} $\Phi\colon\R^d\to\R\cup\{\infty\}$, meaning:
\begin{enumerate}
\item (Convexity) $\Phi$ is proper, closed, and strictly convex.
\item (Essential smoothness) The domain
	of $\Phi$ has nonempty interior,	$\Phi$ is differentiable on $\inter(\dom \Phi)$, and for any sequence $\{x_k\}\subset \inter(\dom \Phi)$ converging to a boundary point of $\dom \Phi$, it must be the case that $\|\nabla \Phi(x_k)\|\to\infty$.
\end{enumerate}

Typical examples of Legendre functions are the squared Euclidean norm $\Phi(x)=\frac{1}{2}\|x\|^2_2$, the Shannon entropy $\Phi(x)=\sum_{i=1}^d x_i\log(x_i)$ with $\dom \Phi=\R^d_+$, and the Burge function $\Phi(x)=-\sum_{i=1}^d\log(x_i)$ with $\dom \Phi= \R^d_{++}$. For more examples, we refer the reader to the articles \cite{AusTeb,Baus_bor_breg,Eck_breg,teb_breg_2} and the recent survey \cite{Teboulle2018}.  

We will often use the observation that the subdifferential of a Legendre function $\Phi$ is empty on the boundary of its domain \cite[Theorem 26.1]{rockafellar}:
$$\partial \Phi(x)=\emptyset \qquad \textrm{for all }x\notin \inter(\dom \Phi).$$
The Legendre function $\Phi$ induces the {\em Bregman divergence}
\[ D_\Phi (y,x) := \Phi(y) - \Phi(x) - \la \nabla \Phi (x), y-x \ra,\]
for all $x \in \inter(\dom \Phi),~y\in \dom \Phi$. Notice that since $\Phi$ is strictly convex, equality $D_\Phi (y,x)=0$ holds for some $x,y\in \inter(\dom \Phi)$ if and only if $y=x$. Analysis of algorithms based on the Bregman divergence typically relies on the following three point inequality; see e.g. \cite[Property 1]{tseng}.

\begin{lemma}[Three point inequality]\label{lem:threepoint}
	Consider a closed convex function $g\colon\R^d\to\R\cup\{+\infty\}$ satisfying $\ri(\dom  g)\subset \inter(\dom \Phi)$.
	Then for any point $z\in \inter(\dom \Phi)$, any minimizer $z^+$ of the problem
\begin{equation*} 
	\min_{x}~ g(x)+D_{\Phi}(x,z),
\end{equation*}
	lies in $\inter(\dom \Phi)$, is unique, and satisfies the inequality:
	$$g(x)+D_{\Phi}(x,z)\geq g(z_+)+D_{\Phi}(z_+,z)+D_{\Phi}(x,z_+)\qquad\forall x\in \dom \Phi.$$
\end{lemma}

Recall that a function $f\colon\R^d\to\R\cup\{\infty\}$ is called $\rho$-weakly convex if the perturbed function $f+\frac{\rho}{2}\|\cdot\|^2_2$ is convex \cite{Nurminskii1973}. By analogy, 
we will say that $f$ is {\em $\rho$-weakly convex relative to} $\Phi$ if the perturbed function $f + \rho\Phi$ is convex. This notion is closely related to the relative smoothness condition introduced in \cite{rel_smooth_freund,descentBBT}. 

Relative weak convexity, like its classical counterpart, can be caracterized through generalized derivatives. Recall that the {\em Fr\'{e}chet subdifferential} of a function $f$ at a point $x\in \dom f$, denoted $\hat{\partial} f(x)$, consists of all vectors $v\in\R^d$ satisfying
$$f(y)\geq f(x)+\langle v,y-x\rangle+o(\|y-x\|)\qquad \textrm{as }y\to x.$$
The {\em limiting subdifferential} of $f$ at $x$, denoted $\partial f(x)$, consists of all  vectors $v\in\R^d$ such that there exist sequences  $x_k\in\R^d$ and $v_k\in \hat \partial f(x_k)$ satisfying $(x_k,f(x_k),v_k)\to(x,f(x),v)$.

\begin{lemma}[Subdifferential characterization]\label{lem:char_subdiff} {\hfill \\ } The following are equivalent for any locally Lipschitz function $f\colon \Rd \to\R$.
	\begin{enumerate}
		\item\label{it:1} The function $f$ is $\rho$-weakly convex relative to $\Phi$.
		\item\label{it:2} For any $x\in \inter(\dom \Phi),y\in \dom \Phi$ and any $v\in\hat{\partial} f(x)$, the  inequality holds:
		\begin{equation}\label{eqn:lower_bound_breg}
		f(y)\geq f(x)+\langle v,y-x\rangle-\rho D_{\Phi}(y,x).
		\end{equation}
		\item\label{it:3} For any $x\in\inter(\dom \Phi)\cap \dom(\nabla f)$, and any $y\in  \dom \Phi$, the inequality holds:
		\begin{equation}\label{eqn:lower_bound_breg2}
		f(y)\geq f(x)+\langle \nabla f(x),y-x\rangle-\rho D_{\Phi}(y,x).
		\end{equation}
	\end{enumerate}
			If $f$ and $\Phi$ are $C^2$-smooth on $\inter(\dom \Phi)$, then the three properties above are all equivalent to 
			\begin{equation}\label{eqn:second_orderchar}
			\nabla^2 f(x)\succeq -\rho \nabla^2 \Phi(x)\qquad\qquad \forall x\in \inter(\dom \Phi).	
			\end{equation}
			\end{lemma}
\begin{proof}
	Define the perturbed function $g:=f+\rho\Phi$. We prove the implications $\ref{it:1}\Rightarrow\ref{it:2}\Rightarrow\ref{it:3}\Rightarrow\ref{it:1}$ in order. To this end, suppose \ref{it:1} holds. Since $g$ is convex, the subgradient inequality holds:
	\begin{equation}\label{eqn:subgrad_ineq}
	g(y)\geq g(x)+\langle w,y-x\rangle\qquad  \textrm{ for all } x,y\in \R^d, w\in \partial g(x).
	\end{equation}
	Taking into account that $\Phi$ is differentiable on  $\inter(\dom \Phi)$, we deduce  $\hat{\partial} g(x)=\hat{\partial} f(x)+\rho\nabla \Phi(x)$  for all $x\in  \inter(\dom \Phi)$; see e.g. \cite[Exercise 8.8]{rock_wets}. Rewriting \eqref{eqn:subgrad_ineq} with this in mind immediately yields \ref{it:2}.
	The implication $\ref{it:2}\Rightarrow\ref{it:3}$ is immediate since $\hat \partial f(x)=\{\nabla f(x)\}$, whenever $f$ is differentiable at $x$.
	
	Suppose~\ref{it:3} holds. Fix an arbitrary point $x\in \inter(\dom \Phi)\cap \dom(\nabla f)$. Algebraic manipulation of inequality \eqref{eqn:lower_bound_breg2} yields the equivalent description
	\begin{equation}\label{eqn:rewrtite_subgrad}
	g(y)\geq g(x)+\langle \nabla f(x)+\rho\nabla \Phi(x),y-x\rangle\qquad  \textrm{for all } y\in  \dom \Phi.
	\end{equation}
	It follows that the vector $\nabla f(x)+\rho\nabla \Phi(x)$ lies in the convex subdifferential of $g$ at $x$. Since $f$ is locally Lipschitz continuous,  Rademacher's theorem shows that
$\dom (\nabla f)$ has full measure in $\R^d$. In particular, we deduce from~\eqref{eqn:rewrtite_subgrad} that the convex subdifferential of $g$ is nonempty on a dense subset of $\inter(\dom g)$. Taking limits, it quickly follows that the convex subdifferential of $g$ is nonempty at every point $x\in\inter(\dom g)$
	Using \cite[Exercise 3.1.12(a)]{JL}, we conclude that $g$ is convex on $\inter(\dom g)$. Moreover, appealing to the sum rule \cite[Exercise 10.10]{rock_wets}, we deduce that 
	$\partial g(x)=\emptyset $ for all $x\notin\inter(\dom \Phi)$, since $\partial \Phi(x)=\emptyset$ for all $x\notin\inter(\dom \Phi)$.
Therefore $\partial g$ is a globally monotone map globally. Appealing to \cite[Theorem 12.17]{rock_wets}, we conclude that $g$ is a convex function. Thus item \ref{it:1} holds. This completes the proof of the equivalences $\ref{it:1}\Leftrightarrow\ref{it:2}\Leftrightarrow\ref{it:3}$.

Finally suppose that $f$ and $\Phi$ are $C^2$-smooth on $\inter(\dom \Phi)$. Clearly,  if $f$ is $\rho$-weakly convex relative to $\Phi$, then second-order characterization of convexity of the function $g=f+\rho \Phi$ directly implies \eqref{eqn:second_orderchar}. Conversely, \eqref{eqn:second_orderchar} immediately implies that $g$ is convex on the interior of its domain. The same argument using \cite[Theorem 12.17]{rock_wets}, as in the implication  $\ref{it:3}\Rightarrow\ref{it:1}$, shows that $g$ is convex on all of $\R^d$.
\end{proof}

Notice that the setup so far has not relied on any predefined norm. Let us for the moment make the common assumption that $\Phi$ is  1-strongly convex relative to some norm $\|\cdot\|$ on $\R^d$, which implies 
\begin{equation} \label{eq:mirror_sc}
D_\Phi (y,x)  \ge \tfrac{1}{2} \|y-x\|^2.
\end{equation}
Then using Lemma~\ref{lem:char_subdiff}, we deduce that to check that $f$ is $\rho$-weakly convex relative to $\Phi$, it suffices to verify  the inequality
\begin{equation*}
f(y)\geq f(x)+\langle v,y-x\rangle-\frac{\rho}{2} \|y-x\|^2\qquad \textrm{  for all }x,y\in \dom \Phi, v\in\partial f(x).
\end{equation*}
Recall that a function $f\colon\R^d\to\R$ is called $\rho$-smooth if it satisfies:
$$\|\nabla f(y)-\nabla f(x)\|_*\leq \rho \|y-x\| \qquad \textrm{for all }x,y\in \R^d,$$ 
where $\|\cdot\|_*$ is the dual norm. Thus any $\rho$-smooth function $f$ is automatically $\rho$-weakly convex relative to $\Phi$. 
Our main result will not require $\Phi$ to be 1-strongly convex; however, we will impose this assumption in \Cref{sec:spec_anal_mirror descent} where we  augment our guarantees for the stochastic mirror descent algorithm under a differentiability assumption.

\section{The problem class and the algorithm}\label{sec:problem_class_alg}

We are now ready to introduce the problem class considered in this paper. We will be interested in the optimization problem
\begin{equation}
	\min_{x}~ F(x) := f(x) + r(x)   \label{eq:problem}
\end{equation}
where 
\begin{itemize}
\item $f\colon \Rd \to \R$ is a locally Lipschitz function,
\item $r\colon\R^d\to \R\cup\{\infty\}$ is a closed function having a convex domain,
\item $\Phi\colon\R^d\to\R\cup\{\infty\}$ is some Legendre function satisfying the compatibility conditions:
\begin{equation}\label{eqn:compatibility}
\ri(\dom r) \subseteq \inter(\dom \Phi)\qquad \textrm{and}\qquad\partial (r+\Phi)(x)=\emptyset \textrm{ for all } x\notin\inter(\dom \Phi).
\end{equation}
\end{itemize}
The first two items are standard and mild. The third stipulates that $r$ must be compatible with $\Phi$. In particular, the inclusion $\ri(\dom r) \subseteq \inter(\dom \Phi)$ automatically implies \eqref{eqn:compatibility}, whenever $r$ is convex \cite[Theorem 23.8]{rockafellar}, or more generally whenever a standard qualification condition holds.\footnote{Qualification condition: $\partial^{\infty}r(x)\cap -N_{\scriptsize{\dom} \Phi}(x)=\{0\}$, for all $x\in \dom r\cap \dom \Phi$; see \cite[Proposition 8.12, Corollary 10.9]{rock_wets}.}
To simplify notation, henceforth set $U:=\inter(\dom \Phi)$.

\subsection{Assumptions and the Algorithm}
We now specify the model-based algorithms we will analyze. Fix a probability space $(\Omega, \mathcal{F}, P)$ and equip $\Rd$ with the Borel $\sigma-$algebra. 
To each point $x \in \dom f$ and each random element $\xi \in \Omega$, we associate a stochastic one-sided model $f_x(\cdot, \xi)$ of the function $f$.  Namely, we assume that there exist $\tau,\rho,\lipsymb>0$ satisfying the following properties.

\begin{enumerate}[label=(A\arabic*)]
\item \textbf{(Sampling)} It is possible to generate i.i.d. realizations $\xi_1, \ldots, \xi_T \sim P$\label{assum:sampling}

\item \textbf{(One-sided accuracy)} \label{assum:one_sided}
There is a measurable function $(x,y,\xi) \mapsto f_x(y,\xi)$ defined on $U \times U \times \Omega$ satisfying both 
\[ \Exi{f_x(x, \xi) } = f(x), \qquad \forall x \in U  \cap \dom r\]
and
\begin{equation}\label{eqnupper_bound_model}
 \Exi{f_x(y, \xi) - f(y) }\le \tau D_\Phi (y,x), \qquad \forall x,y\in U   \cap \dom r. 
\end{equation}

\item\label{it:convex_model} \textbf{(Weak convexity of the models)} The functions $f_x(\cdot, \xi) + r(\cdot)$ are $\rho$-weakly convex relative to $\Phi$ for all $x\in U\cap \dom r$, and a.e. $\xi \in \Omega$. 

\item \textbf{(Lipschitzian property)} \label{assum:lip}
There exists a square integrable function $L \colon \Omega\to \R_+$ such that for all $x,y\in U\cap\dom r$, the following inequalities hold:
\begin{align}
f_x(x, \xi) - f_x(y, \xi)  &\le  L(\xi)\sqrt{D_{\Phi}(y,x)}, \label{eqn:lip_funcmodel}\\
\sqrt{\EE_{\xi}\left[L(\xi)^2\right]} &\leq \lipsymb \nonumber.
\end{align}
\end{enumerate}

Some comments are in order. Assumption~\ref{assum:sampling} is standard and is necessary for all sampling based algorithms. Assumption \ref{assum:one_sided}
specifies the accuracy of the models. That is, we require the model in expectation to agree with $f$ at the basepoint, and to globally lower-bound $f$ up to an error controlled by the Bregman divergence. Assumption~\ref{it:convex_model} is very  mild, since in most practical  circumstances the function $f_x(\cdot, \xi) + r(\cdot)$ is convex, i.e. $\rho=0$.
The final Assumption~\ref{assum:lip} controls the order of growth of the individual models $f_x(y,x)$ as the argument $y$ moves away from $x$.

Notice that the assumptions \ref{assum:sampling}-\ref{assum:lip} do not involve any norm on $\R^d$. However, when $\Phi$ is 1-strongly convex relative to some norm, the properties \eqref{eqnupper_bound_model} and \eqref{eqn:lip_funcmodel} are implied by standard assumptions. Namely \eqref{eqnupper_bound_model} holds if the error in the model approximation satisfies 
$$\Exi{f_x(y, \xi) - f(y) }\le \frac{\tau}{2} \|y-x\|^2, \qquad \forall x,y\in U.$$
Similarly \eqref{eqn:lip_funcmodel} will hold as long as for every $x\in U\cap\dom r $ and a.e. $\xi\in \Omega$ the models $f_x(\cdot,\xi)$ are $L(\xi)$-Lipschitz continuous on $U$ in the norm $\|\cdot\|$. The use of the Bregman divergence allows for much greater flexibility as it can, for example, model higher order growth of the functions in question. To illustrate, let us look at the following example where the Lipschitz constant $L(\xi)$ of the models $f_x(\cdot, \xi)$ is bounded by a polynomial.

\begin{example}[Bregman divergence under polynomial growth]\label{ex:ex_polygrowth}{\rm 
Consider a degree $n$ univariate polynomial  $$p(u)=\sum_{i=0}^n a_i u^i,$$
with coefficients $a_i\geq 0$. Suppose now that the one-sided Lipschitz constants of the models satisfy the growth property:
\begin{align*}
\frac{f_{x}(x, \xi) - f_{x}(y, \xi)}{\|x-y\|_2} \leq  L(\xi) \sqrt{\frac{p(\|x\|_2)+p(\|y\|_2)}{2}} \qquad \textrm{ for all distinct } x,y\in\R^d.
\end{align*}

Motivated by \cite[Proposition~5.1]{Lu_mirror_weird}, the following proposition constructs a Bregman divergence that is well-adapted to the polynomial $p(\cdot)$. We defer its proof  to Appendix~\ref{appendix:bregpoly}. In particular, with the choice of the Legendre function $\Phi$ in \eqref{eqn:legendre_func_poly_growth}, the required estimate \eqref{eqn:lip_funcmodel} holds.
\begin{proposition}\label{prop:bregpoly}
Define the convex function
\begin{equation}\label{eqn:legendre_func_poly_growth}
\Phi(x) = \sum_{i=0}^n a_i\left(\frac{3i + 7}{i+2} \right)\|x\|_2^{i+2}.
\end{equation}
Then for all $x, y \in \RR^d$, we have
$$
D_{\Phi}(y, x) \geq  \frac{p(\|x\|_2)+p(\|y\|_2)}{2} \cdot \|x - y\|_2^2,
$$
and therefore the estimate \eqref{eqn:lip_funcmodel} holds.
\end{proposition}
}
\end{example}

The final ingredient we need before stating the algorithm is an estimate on the weak convexity constant of $F$. The following simple lemma shows that Assumptions~\ref{assum:one_sided} and~\ref{it:convex_model} imply that $F$ itself is $(\tau + \rho)$-weakly convex relative to $\Phi$.
\begin{lemma}\label{lem:weak_conv}
	The function $F$ is $(\tau + \rho)$-weakly convex relative to $\Phi$. 
\end{lemma}
\begin{proof}
We first show that the function $g:=F +(\rho+\tau)\Phi$ is convex on $\ri(\dom F)$. To this end, 
fix arbitrary points $x,y \in \ri(\dom g)$, and note the equality $\ri(\dom g)=U\cap\ri(\dom r)$ \cite[Theorem 6.5]{rockafellar}. Choose $\lambda\in (0,1)$ and set $\bar x=\lambda x+(1-\lambda)y$. 
Taking into account \ref{it:convex_model}, we deduce 
\begin{equation}
\label{eqn:long_ver}
\begin{aligned}
g(\bar x)&=f(\bar x) + r(\bar x) +(\rho+\tau)\Phi(\bar x)\\
&=\EE_{\xi}[f_{\bar x}(\bar x,\xi) + r(\bar x) +\rho\Phi(\bar x)]+\tau\Phi(\bar x)\\
&\leq \EE_{\xi} [\lambda (f_{\bar x}(x,\xi) + r( x)+\rho\Phi(x))+(1-\lambda)(f_{\bar x}(y, \xi)  + r(y) +\rho\Phi(y))]+\tau\Phi(\bar x)\\
&=\lambda \EE_{\xi}[f_{\bar x}(x,\xi) + r(x)]+(1-\lambda)\EE_{\xi}[f_{\bar x}(y,\xi) + r(y)]+\tau\Phi(\bar x)+\lambda\rho\Phi(x)+(1-\lambda)\rho\Phi(y)\\
&=\lambda \EE_{\xi}[f_{\bar x}(x,\xi) + r(x)-\tau D_{\Phi}(x,\bar x)]+(1-\lambda)\EE_{\xi}[f_{\bar x}(y,\xi) + r(y)-\tau D_{\Phi}(y,\bar x)]\\
&\qquad+\lambda\tau(\Phi(\bar x)+ D_{\Phi}(x,\bar x))+(1-\lambda)\tau(\Phi(\bar x)+ D_{\Phi}(y,\bar x))+\lambda\rho\Phi(x)+(1-\lambda)\rho\Phi(y).
\end{aligned}
\end{equation}
Now observe
$$\Phi(\bar x)+ D_{\Phi}(x,\bar x)=\Phi(x)-(1-\lambda)\langle \nabla \Phi(\bar x),x-y\rangle,$$
and similarly $$\Phi(\bar x)+ D_{\Phi}(y,\bar x)=\Phi(y)-\lambda\langle \nabla \Phi(\bar x),y-x\rangle.$$
Hence algebraic manipulation of the two equalities above yields the expression $$\lambda\tau(\Phi(\bar x)+ D_{\Phi}(x,\bar x))+(1-\lambda)\tau(\Phi(\bar x)+ D_{\Phi}(y,\bar x))=\lambda\tau\Phi(x)+(1-\lambda)\tau\Phi(y).$$
Continuing with \eqref{eqn:long_ver}, we obtain
\begin{align*}
g(\bar x)&\leq \lambda f(x) + r(x) +(1-\lambda) (f(y) + r(y))\\
&\qquad+\lambda\tau\Phi(x)+(1-\lambda)\tau\Phi(y)+\lambda\rho\Phi(x)+(1-\lambda)\rho\Phi(y)\\
&=\lambda[f(x)+ r(x) +(\tau+\rho)\Phi(x)]+(1-\lambda)[f(y) + r(y)+(\tau+\rho)\Phi(y)]\\
&\leq \lambda g(x)+(1-\lambda) g(y).
\end{align*}
We have thus verified that $g$ is convex on $\ri(\dom g)$.
Appealing to \eqref{eqn:compatibility} and the sum rule \cite[Exercise 10.10]{rock_wets}, we deduce that the subdifferential $\partial g(x)$ is empty at every point in $x\notin \ri(\dom g)$, and therefore $\partial g$ is a globally monotone map. Using \cite[Theorem 12.17]{rock_wets}, we conclude that $g$ is a convex function, as needed.
\end{proof}

In light of Lemma~\ref{lem:weak_conv}, we also make the following additional assumption on the solvability of the Bregman proximal subproblems.

\begin{enumerate}[label=(A\arabic*)]
\item[(A5)]\label{it:solvability} \textbf{(Solvability)} The convex problems 
$$\min_y\left\{F(y)+\frac{1}{\lambda}D_{\Phi}(y,x)\right\} \qquad \textrm{and}\qquad~\min_y\left\{f_x(y, \xi)+r(y)+\frac{1}{\lambda}D_{\Phi}(y,x)\right\},$$
admit a minimizer for any $\lambda<(\tau+\rho)^{-1}$, any $x\in U$, and a.e. $\xi\in \Omega$.\footnote{Note the minimizers are automatically unique by Lemma~\ref{lem:threepoint}} The minimizers vary measurably in $(x,\xi)\in  U\times \Omega$.
\end{enumerate}

Assumption (A5) is very mild. In particular, it holds automatically if $(i)$ $\Phi$ is strongly convex with respect to some norm, or if $(ii)$ the functions $f_x(\cdot, \xi)+r(\cdot)+\rho D_{\Phi}(\cdot,x)$ and $F+(\tau+\rho)\Phi$ are bounded from below and $\Phi$ has bounded sublevel sets \cite[Lemma 2.3]{Teboulle2018}.

We are now ready to state the stochastic model-based algorithm we analyze---Algorithm~\ref{alg:stoc_prox}.
\smallskip

	\begin{algorithm}[H]
		\KwData{$x_0\in U\cap\dom r$,  real $\lambda < (\tau + \rho)^{-1}$, a nonincreasing sequence $\{\eta_t\}_{t\geq 0} \subseteq (0,\lambda)$, and iteration count $T$.}
		{\bf Step } $t=0,\ldots,T$:\\		
		\begin{equation*}\left\{
		\begin{aligned}
		&\textrm{Sample } \xi_t \sim P\\
		& \textrm{Set } x_{t+1} = \argmin_{x}~ \left\{f_{x_t}(x,\xi_t)+r(x) + \tfrac{1}{\eta_t} D_\Phi(x, x_t)\right\}
		\end{aligned}\right\},
		\end{equation*}
		Sample $t^*\in \{0,\ldots,T\}$ according to the discrete probability distribution
		$$\mathbb{P}(t^*=t)\propto \frac{\eta_t}{1-\eta_t \rho}.$$
		{\bf Return} $x_{t^*}$		
		\caption{Stochastic Model Based Minimization
		}
		\label{alg:stoc_prox}
	\end{algorithm}

\subsection{Examples}\label{sec:examples}
Before delving into the convergence analysis of Algorithm~\ref{alg:stoc_prox}, in this section we illustrate the algorithmic framework on four examples. In all cases, assumptions 
\ref{assum:sampling} and (A5) are self-explanatory. Therefore, we only focus on verifying \ref{assum:one_sided}-\ref{assum:lip}. For simplicity, we also assume that $r(\cdot)$ is convex in all examples.

\paragraph{Stochastic Bregman-proximal point.}
Suppose that the models $(x,y,\xi) \mapsto f_x(y,\xi)$ satisfy $$\EE_{\xi}[f_x(y,\xi)]=f(y)\qquad \forall x,y\in U\cap \dom r.$$
With this choice of the models, Algorithm~\ref{alg:stoc_prox} becomes the stochastic Bregman-proximal point method. Analysis of the deterministic version of the method for convex problems goes back to \cite{breg_marc,CZ_prox,Eck_breg}.  Observe that Assumption~\ref{assum:one_sided} holds trivially. 
Assumption~\ref{it:convex_model} and Assumption~\ref{assum:lip} should be verified in particular circumstances, depending on how the models are generated. In particular, one can verify Assumption~\ref{assum:lip} under polynomial growth of the Lipschitz constant, by appealing to Example~\ref{ex:ex_polygrowth}.

\paragraph{Stochastic mirror descent.}
Suppose that the models $(x,y,\xi) \mapsto f_x(y,\xi)$ are given by 
$$f_x(y,\xi)=f(x)+\langle G(x,\xi),y-x\rangle,$$
for some  measurable mapping $G \colon U \times \Omega \rightarrow \RR^d$ satisfying  $\EE_{\xi}[G(x,\xi)]\in \partial f(x)$ for all $x\in U\cap \dom r$. Algorithm~\ref{alg:stoc_prox} then becomes the stochastic mirror descent algorithm, classically studied in \cite{nem_yud,beckteb} in the convex setting and more recently analyzed in \cite{rel_smooth_freund,descentBBT,Lu_mirror_weird} under convexity and relative continuity assumptions. Assumption~\ref{assum:one_sided} simply says that $f$ is $\tau$-weakly convex relative to $\Phi$, while Assumption~\ref{it:convex_model} holds trivially with $\rho=0$.
Assumption~\ref{assum:lip} is directly implied by the relative continuity condition of Lu~\cite{Lu_mirror_weird}. Namely it suffices to assume that there is a square integrable function $L\colon\Omega \rightarrow \RR_{++}$ satisfying
$$
\|G(x, \xi)\|_\ast \leq L(\xi)\frac{\sqrt{D(y, x)}}{\|y - x\|} \qquad\forall x,y\in U,
$$
 where $\|\cdot\|$ is an arbitrary norm on $\R^d$, and $\|\cdot\|_\ast$ is the dual norm. We refer to \cite{Lu_mirror_weird} for more details on this condition and examples.

\paragraph{Gauss-Newton method with Bregman regularization.}
In the next example, suppose that $f$ has the composite  form 
$$f(x)=\mathbb{E}_{\xi}[h(c(x,\xi),\xi)],$$
for some measurable function $h(y,\xi)$ that is convex in $y$ for a.e. $\xi\in \Omega$ and a measurable map $c(x,\xi)$ that is $C^1$-smooth in $x$ for a.e. $\xi\in \Omega$. We may then use the convex models
$$f_{x}(y,\xi):=h\left(c(x,\xi)+\nabla c(x,\xi)(y-x) ,\xi),\xi\right),$$
which automatically satisfy~\ref{it:convex_model} with $\rho=0$.
Algorithm~\ref{alg:stoc_prox} then becomes a stochastic Gauss-Newton method with Bregman regularization.

In the Euclidean case $\Phi=\frac{1}{2}\|\cdot\|^2$, the method reduces to the stochastic prox-linear algorithm, introduced  in \cite{duchi_ruan} and further analyzed in \cite{stochastic_subgrad}. The deterministic prox-linear method has classical roots, going back at least to \cite{burke_com,fletcher_back,pow_glob}, while a more modern complexity theoretic perspective appears  in \cite{prox,prox_error,composite_cart,nest_GN,prox_dim_court}. 
Even in the deterministic setting, to make progress, one typically assumes that $h$ and $\nabla c$ are globally Lipschitz.
 More generally and in line with our current work,  one may introduce a different Legendre function $\Phi$. For example, in the case of polynomial growth, the following propositions  construct Legendre functions that are compatible with Assumptions~\ref{assum:one_sided} and~\ref{assum:lip}.  We defer their proofs to Appendix~\ref{appendix:prox_linear_accuracy}.  In the two propositions, we assume that the outer functions $h(\cdot,\xi)$ are globally Lipschitz, while the inner maps $c(\cdot,\xi)$ may have a high order of growth. It is possible to also analyze the setting when $h(\cdot,\xi)$ has polynomial growth, but the resulting statements  and assumptions become much more cumbersome; we therefore omit that discussion.

\begin{proposition}[Satisfying~\ref{assum:one_sided}] \label{prop:prox_linear_accuracy}
Suppose  there are square integrable functions $L_1, L_2 \colon \Omega \rightarrow \RR_+$ and a univariate polynomial $p(u)=\sum_{i=0}^n a_i u^i$ with nonnegative coefficients satisfying
\begin{align*}
\frac{|h(v,\xi) - h(w,\xi)|}{\|v - w\|_2} &\leq L_1(\xi) \qquad \forall v\neq w,\\
\frac{\|\nabla c(x,\xi) - \nabla c(y,\xi) \|_{{\rm op}}}{\|x - y\|_2} &\leq L_2(\xi) (p(\|x\|_2) + p(\|y\|_2))\qquad \forall x\neq y.
\end{align*}
Define the Legendre function
$
\Phi(x) := \sum_{i = 0}^n \frac{a_i(3i + 7)}{i+2}\|x\|_2^{i+2}.
$
Then  assumption~\ref{assum:one_sided} holds with $\tau := \tfrac{4}{3}\EE\left[L_1(\xi)L_2(\xi)\right]$.

\end{proposition}

\begin{proposition}[Satisfying~\ref{assum:lip}]\label{prop:prox_linear_lipschitz}
Suppose there are square integrable functions $L_1, L_2 \colon \Omega \rightarrow \RR_+$ and a univariate polynomial $q(u)=\sum_{i=0}^n b_i u^i$ with nonnegative coefficients satisfying
\begin{align*}
\frac{|h(v,\xi) - h(w,\xi)|}{\|v - w\|_2} &\leq L_1(\xi) \qquad \forall v\neq w,\\
   \|\nabla c(x, \xi)\|_{\rm op} &\leq L_2(\xi) \cdot \sqrt{q(\|x\|_2)}\qquad \forall x,\xi.
\end{align*}
Then with the Legendre function $\Phi(x) = \sum_{i=0}^n \frac{b_i}{i+2}\|x\|_2^{i+2}$, assumption~\ref{assum:lip} holds with $L(\xi) =\sqrt{2} L_1(\xi)L_2(\xi)$.
\end{proposition}

To construct a Bregman function compatible with both \ref{assum:one_sided} and~\ref{assum:lip} simultaneously, one may simply add the two Legendre functions constructed in Propositons~\ref{prop:prox_linear_accuracy} and~\ref{prop:prox_linear_lipschitz}.

\paragraph{Stochastic saddle point problems.} 

As the final example, suppose that $f$ is given in the stochastic conjugate form
$$
f(x) = \EE\left[ \sup_{w \in W} g(x, w, \xi) \right],
$$
where $W$ is some auxiliary set and $g \colon \RR^d \times W \times \Omega \rightarrow \RR$ is some function. Thus we are interested in solving the stochastic saddle-point problem 
\begin{equation}\label{eqn:stoch_saddle}
\inf_{x}~\EE\left[ \sup_{w \in W} g(x, w, \xi) \right]+r(x).
\end{equation}
Such problems appear often in data science, where the variation of $w$ in the ``uncertainty set'' $W$  makes the loss function robust. One popular example is adversarial training \cite{adverse_training}. In this setting, we have $g(x,w,\xi)=\mathcal{L}(x+w,y,\xi)$, where $\mathcal{L}(\cdot,\cdot)$ is a loss function, $y$ encodes the observed data, and $w$ varies over some uncertainty set $W$, such as an $\ell_p$-ball. 

In order to apply our algorithmic framework, we must have access to stochastic one-sided models $f_{x}(\cdot,\xi)$ of $f$. It is quite natural to construct such models by using one-sided  stochastic models  $g_x(\cdot,w,\xi)$ of $g$. Indeed, it is appealing to simply set 
\begin{equation}\label{eqn:robustiviszing}
f_{x}(y, \xi) = g_{x}(y, \widehat w(x, \xi), \xi) \qquad \textrm{for any }\qquad \widehat{w}(x,\xi)\in\argmax_w~ g_{x}(x,w,\xi). 
\end{equation}
All of the model types in the previous examples could now serve as the models $g_{x}(\cdot, w, \xi)$, provided they meet the conditions outlined below.

Formally, to ensure that \ref{assum:sampling}-(A5) hold for the models $f_{x}(y, \xi)$, we must make the following assumptions:
\begin{enumerate}
\item The mapping $(x, \xi) \rightarrow \sup_{w\in W} g(x, w, \xi)$ is measurable and has finite first moment for every fixed $x \in U\cap\dom r$.
\item \label{item:saddle_weak_convex} The function $g_{x}(\cdot, w, \xi)$ is $\rho$-weakly convex  relative to $\Phi$, for every fixed $x \in U\cap\dom r$, $w \in W$, and a.e. $\xi \in \Omega$.
\item\label{item:saddle_maximizer} There exists a  mapping $\widehat w\colon U\times \Omega \rightarrow \RR^m$ satisfying $$\widehat w(x, \xi) \in \argmax_w g_x(x, w,\xi),$$ for all $x \in U\cap\dom r$ and a.e. $\xi \in \Omega$ with the property that the functions $(x, y, \xi) \mapsto g_{x}(y, \widehat w(x, \xi), \xi)$ and $(x, y, \xi) \mapsto g(y, \widehat w(x, \xi), \xi)$ are measurable.
\item \label{item:saddle_equals} For all $x,y \in U\cap\dom r$, we have $$\mathbb{E}_{\xi}\left[g_x(x, \widehat w(x,\xi), \xi)\right] = \mathbb{E}_{\xi}\left[g(x,\widehat w(x,\xi),\xi)\right]$$
and $$\EE\left[ g_{x}(y, \widehat w(x, \xi) , \xi) -  g(y, \widehat w(x, \xi) , \xi)\right] \leq \tau D_\Phi(y, x).$$
\item \label{item:saddle_lipschitz} There exists a square integrable function $L\colon\Omega\to\R_+$ such that $$g_{x}(x, \widehat w(x, \xi) , \xi) - g_{x}(y, \widehat w(x, \xi) , \xi)  \leq L(\xi) \sqrt{D_{\Phi}(y, x)}, \qquad \textrm{for all } x,y \in U\cap\dom r.$$ 
\end{enumerate}
Given these assumptions, let us define $f_{x}(y, \xi)$ as in \eqref{eqn:robustiviszing}
We now verify properties~\ref{assum:one_sided}-\ref{assum:lip}. Property~\ref{assum:one_sided} follows from Property~\ref{item:saddle_equals}, which implies that
$
\EE\left[ f_{x}(x, \xi) \right] = f(x)
$
and 
\begin{align*}
\EE_{\xi}\left[f_{x}(y, \xi) - f(y) \right]&= \EE_{\xi}\left[g_{x}(y, \widehat w(x, \xi), \xi) - \sup_{w\in W} g(y, w, \xi) \right]\\
&\leq \EE_{\xi}\left[g_{x}(y, \widehat w(x, \xi), \xi) -  g(y, \widehat w(x, \xi), \xi) \right]\\
&\leq \tau D_\Phi(y, x).
\end{align*}
Property~\ref{it:convex_model} follows directly from Property~\ref{item:saddle_weak_convex}. Finally, \ref{assum:lip} follows from Property~\ref{item:saddle_lipschitz}.

\section{Stationarity measure}\label{sec:conv_guarant}
In this section, we introduce a natural stationarity measure that we will use to describe the convergence rate of Algorithm~\ref{alg:stoc_prox}. The stationarity measure is simply the size of the gradient of an appropriate smooth approximation of the problem \eqref{eq:problem}. This idea is completely analogous to the Euclidean setting \cite{stochastic_subgrad,subgrad_pnas}.
Setting the stage, for any $\lambda >0$, define the {\em $\Phi$-envelope}
\[ F^\Phi_\lambda (x) := \inf_{y} \left\{ f (y) + \frac{1}{\lambda}D_\Phi(y,x) \right\}, \]
and the associated {\em $\Phi$-proximal map}
\[ \prox_{\lambda f}^\Phi (x) :=   \argmin_y \left\{ F(y) + \frac{1}{\lambda} D_\Phi(y,x) \right\} . \]
Note that in the Euclidean setting $\Phi=\frac{1}{2}\|\cdot\|^2$, these two constructions reduce to the standard Moreau envelope and the proximity map; see for example the monographs \cite{rock_wets,Parikh_survey} or the note \cite{dima_news} for recent perspectives. 

We will measure the convergence guarantees of Algorithm~\ref{alg:stoc_prox} based on the rate at which the quantity
\begin{equation}\label{eqn:stat_quant}
\EE[D_{\Phi}\left(\prox_{\lambda F}^\Phi (x_{t^*}),x_{t^*}\right)]
\end{equation}
tends to zero for some fixed $\lambda>0$. The significance of this quantity becomes apparent after making slightly stronger assumptions on the Legendre function $\Phi$. In this section only, suppose that $\Phi\colon\R^d\to\R\cup\{+\infty\}$ is 1-strongly convex with respect to some norm $\|\cdot\|$ and that $\Phi$ is twice differentiable at every point in $\inter(\dom \Phi)$. With these assumptions, the following result shows that the $\Phi$-envelope is differentiable, with a meaningful gradient. Indeed, this result follows quickly from  \cite{breg_baus}. For the sake of completeness, we present a self-contained argument in Appendix~\ref{sec:app_proof_diff_breg}.

\begin{theorem}[Smoothness of the $\Phi$-envelope] \label{thm:env_diff}
For any positive $\lambda< (\tau+\rho)^{-1}$, the envelope $F_\lambda^\Phi$ is differentiable at any point $x\in\inter(\dom \Phi)$ with gradient given by
	\[ \nabla F^\Phi_\lambda (x) := \frac{1}{\lambda} \nabla^2 \Phi (x) \left(x - \prox_{\lambda F}^\Phi (x)\right). \]
\end{theorem}

\noindent
In light of \Cref{thm:env_diff}, for any point $x\in \inter(\dom \Phi)$, we may define the local norm 
$$\|y\|_x:=\left\|\nabla^2 \Phi(x)y\right\|_*.$$
Then a quick computation shows that the dual norm is given by 
$$\|v\|_x^*=\left\|\nabla^2 \Phi(x)^{-1}v\right\|.$$
Therefore appealing to Theorem~\ref{thm:env_diff}, for any positive $\lambda< (\tau+\rho)^{-1}$ and $x\in \inter(\dom \Phi)$ we obtain the estimate
$$\sqrt{D_{\Phi}\left(\prox_{\lambda F}^\Phi (x),x\right)}\geq \frac{\lambda}{\sqrt{2}}\|\nabla F^\Phi_\lambda (x)\|^*_x.$$
Thus the square root of the Bregman divergence, which we will show tends to zero along the iterate sequence at a controlled rate,  bounds the local norm of the  gradient $\nabla F^\Phi_\lambda$.

\section{Convergence analysis}\label{sec:conv_anal}
We now present convergence analysis of Algorithm~\ref{alg:stoc_prox} under Assumptions \ref{assum:sampling}-(A5).
Henceforth, let $\{x_t\}_{t\geq 0}$ be the iterates generated by Algorithm~\ref{alg:stoc_prox} and let $\{\xi_t\}_{t\geq 0}$ be the corresponding samples used. For each index $t\geq 0$, define the Bregman-proximal point 
$$\hat x_t=\prox_{\lambda F}^\Phi (x_t).
$$ 
To simplify notation, we will use the symbol $\mathbb{E}_{t}[\cdot]$ to denote the expectation conditioned on all the realizations $\xi_0,\xi_1,\ldots, \xi_{t-1}$.
The entire argument of Theorem~\ref{thm:convergence}---our main result---relies on the following lemma.

\begin{lemma} \label{lem:3pt}
For each iteration $t\geq 0$, the iterates of Algorithm~\ref{alg:stoc_prox} satisfy
\[  
 \Et{D_\Phi (\hat x_t, x_{t+1})} \le
\tfrac{1 + \eta_t \tau - \eta_t/\lambda}{1-\eta_t\rho} D_\Phi (\hat x_t, x_t)  +\tfrac{ (\lipsymb \eta_t)^2}{4(1-\eta_t\rho)} + \tfrac{\eta_t}{1-\eta_t\rho}\Et{r(x_{t}) - r(x_{t+1})}. 
\]
\end{lemma}

\begin{proof}
Taking into account assumption~\ref{it:convex_model}, we may apply the three point inequality in \Cref{lem:threepoint} with the convex function $g=f_{x_t}(\cdot,\xi_t)+r(\cdot)+\rho D_{\Phi}(\cdot,x_t)$ and with $(\tfrac{1}{\eta_t}-\rho )D_{\Phi}(\cdot,x_t)$ replacing the Bregman divergence. Thus for any point $x\in \inter(\dom \Phi)$, we obtain the estimate
\begin{equation} \label{eq:3pt}
 f_{x_t}(x, \xi_t) + r(x) + \frac{1}{\eta_t} D_\Phi (x, x_t) 
 \ge f_{x_t}(x_{t+1}, \xi_t)  + r(x_{t+1}) + \frac{1}{\eta_t} D_\Phi (x_{t+1}, x_t) + \left( \frac{1}{\eta_t} - \rho\right) D_\Phi (x, x_{t+1}) .
\end{equation}
Setting $x = \hat x_t$, rearranging terms, and taking expectations, we deduce
\begin{equation} \label{eq:init_bound}
 \begin{aligned}
 \mathbb{E}_{\xi}[f_{x_t}(\hat x_t, \xi_t) + &r(\hat x_t)-  f_{x_t}(x_{t+1}, \xi_t) - r(x_{t+1}) ] \\
 &\geq  \frac{1}{\eta_t} \Et{ (1-\eta_t\rho) D_\Phi (\hat x_t, x_{t+1}) - D_\Phi (\hat x_t, x_{t}) + D_\Phi (x_{t+1}, x_{t}) }.
 \end{aligned}
 \end{equation}
We seek to upper bound the left-hand-side of \eqref{eq:init_bound}. Using assumptions \ref{assum:one_sided} and \ref{assum:lip}, we obtain:
\begin{equation}\label{eqn:weird_multi_later_used}
\begin{aligned}
& \Et{  f_{x_t} (\hat x_t , \xi_t) - f_{x_t}( x_{t+1}, \xi_t)  }\\
 &\hspace{20pt}\le \Et{  f_{x_t} (\hat x_t , \xi_t) - f_{x_t}( x_{t}, \xi_t)  + L(\xi) \sqrt{ D_{\Phi}(x_{t+1},x_t)} } \\
 &\hspace{20pt}= \Et{  f_{x_t} (\hat x_t , \xi_t) - f(\hat x_t)}   + \Et { L(\xi)\sqrt{ D_{\Phi}(x_{t+1},x_t)} } - f(x_t) + f(\hat x_t) \\
 &\hspace{20pt}\le \tau D_\Phi (\hat x_t, x_t) + \Et { L(\xi)\sqrt{ D_{\Phi}(x_{t+1},x_t)} } - f(x_t) + f(\hat x_t).
\end{aligned}
\end{equation}
By the definition of $\hat x_t$ as the Bregman-proximal point, we have
\begin{equation}\label{eqn:prox_point_defn_proof}
 f(\hat x_t) + r(\hat x_t) + \frac{1}{\lambda} D_\Phi(\hat x_t, x_t) \le f(x_{t}) + r(x_{t}) . 
\end{equation}
The right hand side of \eqref{eq:init_bound} is thus upper bounded by
\begin{align*}
& \tau D_\Phi (\hat x_t, x_t) + \Et {  L(\xi) \sqrt{ D_{\Phi}(x_{t+1},x_t)}  - f(x_t) - r(x_{t+1} )}+ f(\hat x_t) + r(\hat x_t) \\
 &\le \tau D_\Phi (\hat x_t, x_t) + \Et {  L(\xi) \sqrt{ D_{\Phi}(x_{t+1},x_t)} + (r(x_{t}) - r(x_{t+1})) } + f(\hat x_t) + r(\hat x_t) - f(x_{t}) - r(x_{t}) \\
 &\le \left(\tau - \frac{1}{\lambda}\right)D_\Phi (\hat x_t, x_t) + \Et {  L(\xi) \sqrt{ D_{\Phi}(x_{t+1},x_t)}   +  (r(x_{t}) - r(x_{t+1})) }
\end{align*}
where the last inequality follows from \eqref{eqn:prox_point_defn_proof}.
Combining this estimate with \eqref{eq:init_bound}, we obtain
\begin{align*}
&\frac{1}{\eta_t} \Et{ (1-\eta_t\rho) D_\Phi (\hat x_t, x_{t+1}) - D_\Phi (\hat x_t, x_{t}) + D_\Phi (x_{t+1}, x_{t}) }\\
 &\hspace{20pt}\le
 \left( \tau - \frac{1}{\lambda} \right) D_\Phi (\hat x_t, x_t)  +  \Et {  L(\xi) \sqrt{ D_{\Phi}(x_{t+1},x_t)}+(r(x_{t}) - r(x_{t+1})) }  .
 \end{align*}
  Multiplying through by $\eta_t$ and rearranging yields
 \begin{equation}\label{eqn:rand_ass_eqn}
 \begin{aligned}
 &(1-\eta_t\rho) \Et{ D_\Phi (\hat x_t, x_{t+1}) } \\
 &\hspace{20pt}\le
\left( 1 +  \eta_t \tau - \frac{\eta_t}{\lambda} \right) D_\Phi (\hat x_t, x_t)  +   \Et { \eta_t L(\xi) \sqrt{ D_{\Phi}(x_{t+1},x_t)}  -  { D_\Phi (x_{t+1}, x_{t}) }}\\
&\hspace{40pt} + \eta_t\Et{r(x_{t}) - r(x_{t+1})}.
 \end{aligned}
 \end{equation}
Now define $\gamma := \sqrt{\Et {   D_{\Phi}(x_{t+1},x_t)}}$. 
Note that Cauchy-Schwarz implies $$\Et { \eta_t L(\xi) \sqrt{ D_{\Phi}(x_{t+1},x_t)}} \leq \eta_t \lipsymb\gamma.$$ Using this estimate in \eqref{eqn:rand_ass_eqn}, we obtain
\begin{align*}
 (1-\eta_t\rho) \Et{ D_\Phi (\hat x_t, x_{t+1}) }
 &\le
 \left( 1 + \eta_t \tau - \frac{\eta_t}{\lambda} \right) D_\Phi (\hat x_t, x_t)  +   \eta_t\lipsymb \gamma -   \gamma^2 \\
 &\hspace{20pt} +  \eta_t\Et{r(x_{t}) - r(x_{t+1})}.
\end{align*}
 Maximizing the right hand side in $\gamma$ (i.e. taking $\gamma = \frac{\lipsymb\eta_t}{2}$), yields the guarantee
  \[ 
(1-\eta_t\rho) \Et{ D_\Phi (\hat x_t, x_{t+1}) }
 \le
 \left( 1 + \eta_t \tau - \frac{\eta_t}{\lambda} \right) D_\Phi (\hat x_t, x_t)  +\frac{(\lipsymb\eta_t)^2}{4} + \eta_t\Et{r(x_{t}) - r(x_{t+1})} .
 \]
Dividing through by $1-\eta_t\rho$ completes the proof.
\end{proof}

\noindent
We can now prove our main theorem.

\begin{theorem}[Convergence rate] \label{thm:convergence}
The point $x_{t^*}$ returned by Algorithm~\ref{alg:stoc_prox} satisfies:
\begin{align*}
&\EE\left[D_{\Phi}\left(\prox_{\lambda F}^\Phi (x_{t^*}),x_{t^*}\right)\right]\\
&\hspace{20pt}\le \frac{\lambda^2}{1 - \lambda (\tau + \rho)} \left( 
\frac{F_\lambda^\Phi(x_0) - \min F}{\sum_{t=0}^{T} \frac{\eta_t}{1-\eta_t \rho}} + \frac{ \lipsymb^2\sum_{t=0}^{T} \frac{ \eta_t^2 }{4\lambda(1 -\eta_t \rho)} }{\sum_{t=0}^{T} \frac{\eta_t}{1-\eta_t \rho}} +  \frac{\tfrac{\eta_0}{\lambda(1-\eta_0\rho)} (r(x_0)-\inf r) }{\sum_{t=0}^{T} \frac{\eta_t}{1-\eta_t \rho}} \right).
\end{align*}

\end{theorem}
\begin{proof}
Using the definitions of $x_{t+1}$ and $\hat x_t$ along with \Cref{lem:3pt}, we obtain
\begin{align*}
\Et{F^\Phi_\lambda(x_{t+1})}
&\le \Et{ F(\hat x_t) + \frac{1}{\lambda} D_\Phi (\hat x_t, x_{t+1}) } \\
&\le \Et { F(\hat x_t) + \frac{1}{\lambda(1- \eta_t \rho) } \left( \left( 1 +  \eta_t \left(\tau - \frac{1}{\lambda}\right) \right) D_\Phi (\hat x_t, x_t)  + \frac{(\lipsymb \eta_t)^2}{4}  \right) } \\
&\hspace{20pt} + \frac{\eta_t}{\lambda(1-\eta_t\rho)}\Et{(r(x_{t}) - r(x_{t+1}))}\\
& = F^\Phi_\lambda(x_t) + \frac{\eta_t}{\lambda} \left(\frac{ \tau +\rho - 1/\lambda}{1 - \eta_t \rho} \right) D_\Phi (\hat x_t, x_t)  +\frac{(\lipsymb \eta_t)^2}{4\lambda(1 - \eta_t \rho)} \\
&\hspace{20pt} + \frac{\eta_t}{\lambda(1-\eta_t\rho)}\Et{r(x_{t}) - r(x_{t+1})}.
\end{align*}
Recursing and applying the tower rule for expectations, we obtain
\begin{equation}\label{eqn:key_recurs_est}
\begin{aligned}
\E{F_\lambda^\Phi(x_{T+1})}
&\le 
F_\lambda^\Phi(x_0) + \sum_{t=0}^T \left(   \frac{\eta_t}{\lambda} \left(\frac{ \tau +\rho - 1/\lambda}{1 - \eta_t \rho} \right) \mathbb{E}[D_\Phi (\hat x_t, x_t)]  +\frac{(\lipsymb \eta_t)^2}{4\lambda(1 - \eta_t \rho)}  \right) \\
&\hspace{20pt} +\sum_{t=0}^T \frac{\eta_t}{\lambda(1-\eta_t\rho)}\E{r(x_{t}) - r(x_{t+1})}.
\end{aligned}
\end{equation}
Taking into account that $\eta_t$ is nonincreasing yields the inequality
$$\sum_{t=0}^T \tfrac{\eta_t}{\lambda(1-\eta_t\rho)}{(r(x_{t}) - r(x_{t+1}))}\leq \tfrac{\eta_0}{\lambda(1-\eta_0\rho)} (r(x_0) - \inf r).$$
See the auxiliary Lemma~\ref{lem:aux-lem} for a verification. Combining this bound with \eqref{eqn:key_recurs_est}, using the inequality  $\E{F_\lambda(x_{T+1})} \ge \min F$, and rearranging, we conclude
\begin{align*}
  \frac{1}{\lambda} \left(\frac{1}{\lambda} - \tau - \rho \right)\sum_{t=0}^T \frac{\eta_t}{1-\eta_t\rho} \mathbb{E}[D_\Phi (\hat x_t, x_t)]
&\le 
F_\lambda^\Phi(x_0) - \min F+ \lipsymb^2 \sum_{t=0}^T  \frac{\eta_t^2}{4\lambda(1-\eta_t\rho)} \\
&\hspace{20pt} + \frac{\eta_0}{\lambda(1-\eta_0\rho)} (r(x_0) - \inf r),
\end{align*}
or equivalently
\begin{align*}
\sum_{t=0}^T \frac{\eta_t}{1-\eta_t\rho} \E{D_\Phi (\hat x_t, x_t)}
&\le 
\frac{\lambda^2 (F_\lambda^\Phi(x_0) - \min F)}{ 1-\lambda( \tau +\rho)} 
+ \frac{\lambda^2 \lipsymb^2 }{ 1-\lambda( \tau +\rho)} \sum_{t=0}^T  \frac{\eta_t^2}{4\lambda(1-\eta_t\rho)} \\
&\hspace{20pt} + \frac{\lambda^2\eta_0}{\lambda(1-\lambda( \tau +\rho))(1-\eta_0\rho)}( r(x_0)-\inf r).
\end{align*}
Dividing through by $\sum_{t=0}^{T} \frac{\eta_t}{1-\eta_t \rho}$ and recognizing the left-hand-side as $\EE[D_{\Phi}(\hat x_{t^*},x_{t^*})]$, the result follows.
\end{proof}

As an immediate corollary of \Cref{thm:convergence}, we have the following rate of convergence when the stepsize $\eta_t$ is constant. 

\begin{corollary}[Convergence rate for constant stepsize] \label{cor:convergence_const}
For some $\alpha>0$, set $\eta_t=\frac{1}{\lambda^{-1}+\alpha^{-1}\sqrt{T+1}}$ for all indices $t=1,\ldots, T$. Then the point $x_{t^*}$ returned by Algorithm~\ref{alg:stoc_prox_linear} satisfies:
\begin{align*}
	&\EE\left[D_{\Phi}\left(\prox_{\lambda F}^\Phi (x_{t^*}),x_{t^*}\right)\right]
	\le \frac{\lambda^2( 
		F_\lambda^\Phi(x_0) - \min F) +  \frac{\lambda\lipsymb^2\alpha^2}{4}  +  \frac{\lambda((r(x_0)-\inf r))}{\lambda^{-1}-\rho+\alpha^{-1}}   }{1 - \lambda (\tau + \rho)}\cdot\left(\frac{\lambda^{-1}-\rho}{T+1}+\frac{1}{\alpha\sqrt{T+1}}\right) .
\end{align*}

\end{corollary}

\section{Mirror descent: smoothness and finite variance}\label{sec:spec_anal_mirror descent}

 Assumptions~\ref{assum:sampling}-(A5) are reasonable for the examples described in Section~\ref{sec:examples}, being in line with standard conditions in the literature. However, in the special case that $f$ is smooth and we apply stochastic mirror descent,  Assumption~\ref{assum:lip} is nonstandard. Ideally, one would like to replace this assumption with a bound on the variance of the stochastic estimator of the gradient. In this section, we show that this is indeed possible by slightly modifying the argument in Section~\ref{sec:conv_anal}.
 
 Henceforth, let $\Phi$ be a Legendre function and set $U:=\inter(\dom \Phi)$. In this section, we make the following assumptions:
\begin{enumerate}[label=(B\arabic*)]
\item \textbf{(Sampling)} \label{assum:sampling_2}
 It is possible to generate i.i.d. realizations $\xi_1, \ldots, \xi_T \sim P$

\item \textbf{(Stochastic gradient)} \label{assum:one_sided_2}
There is a measurable mapping $G: U \times \Omega \rightarrow \R^d$ satisfying  
\[ \Exi{G(x, \xi) } = \nabla f(x), \qquad \forall x \in U  \cap \dom r.\]

\item \textbf{(Relative Smoothness)} \label{assum:smooth1}
There exist real $\tau, M \geq 0$, such that
\[-\tau D_\Phi(y, x) \le f(y) - f(x) - \ip{\nabla f(x)}{y-x} \le M D_\Phi (y,x)\qquad \forall x,y\in U  \cap \dom r. \]

\item \textbf{(Relative convexity)}\label{assump_abar4} The function $r$ is $\rho$-weakly convex relative to $\Phi.$

\item \textbf{(Strong convexity of $\Phi$)} \label{assum:sc}
The Legendre function $\Phi$ is 1-strongly convex with respect to some norm $\| \cdot \|$.

\item \textbf{(Finite variance)} \label{assum:smooth2}
The following variance is finite: 
\[ 
\EE_{\xi}\left[ \left\|G(x, \xi) - \nabla f(x) \right\|_*^2 \right] \leq \frac{\sigma^2}{2} < \infty.
\]

\end{enumerate}

Henceforth, we denote by $f_{x}(\cdot, \xi)$ the linear models
\[ f_{x}(y,\xi) := f(x) + \ip{G(x, \xi)}{y-x}, \]
which are built from the stochastic gradient estimator $G$. With this notation in hand, let us compare Assumptions~\ref{assum:sampling_2}-\ref{assum:sc} with Assumptions~\ref{assum:sampling}-\ref{assum:lip}. Evidently, Assumptions~\ref{assum:sampling_2} and~\ref{assum:sampling} are identical. Upon taking expectations, Assumptions~\ref{assum:one_sided_2} and~\ref{assum:smooth1} imply the stochastic one-sided accuracy property~\ref{assum:one_sided} for the linear models $f_x(\cdot,\xi)$, while \ref{assump_abar4} directly implies \ref{it:convex_model}.
Assumptions \ref{assum:sc} and \ref{assum:smooth2} replace the Lipschitzian property~\ref{assum:lip}. 

Finally, we reiterate that the relative smoothness property in~\ref{assum:smooth1} was recently introduced in \cite{descentBBT,rel_smooth_freund} for smooth convex minimization, and extended to smooth nonconvex problems in \cite{nonconv_teb} and to nonsmooth stochastic problems in  \cite{Lu_mirror_weird,rick_rel_smooth}. This property allows for higher order growth than the standard Lipschitz gradient assumptions, commonly analyzed in the literature. We refer the reader to \cite{descentBBT,rel_smooth_freund} for various examples of Bregman functions that arise in applications.

For the sake of clarity, Algorithm~\ref{alg:stoc_prox_linear} instantiates Algorithm~\ref{alg:stoc_prox} in our setting.

	\begin{algorithm}[H]
		\KwData{$x_0\in U\cap \dom r$,  positive $\lambda < (\tau+\rho)^{-1}$, a sequence $\{\eta_t\}_{t\geq 0} \subseteq \left(0,\frac{\lambda}{1+\lambda M} \right)$, and iteration count $T$}
		{\bf Step } $t=0,\ldots,T$:\\		
		\begin{equation*}\left\{
		\begin{aligned}
		&\textrm{Sample } \xi_t \sim P\\
		& \textrm{Set } x_{t+1} = \argmin_{x}~ \left\{ \ip{ G(x_t, \xi_t)  }{x}+r(x) + \tfrac{1}{\eta_t} D_\Phi(x, x_t)\right\}
		\end{aligned}\right\},
		\end{equation*}
		Sample $t^*\in \{0,\ldots,T\}$ according to the discrete probability distribution
		$$\mathbb{P}(t^*=t)\propto \frac{\eta_t}{1-\eta_t\rho}.$$
		{\bf Return} $x_{t^*}$		
		\caption{Mirror descent for smooth minimization
		}
		\label{alg:stoc_prox_linear}
	\end{algorithm}

As in Section~\ref{sec:conv_anal}, the convergence analysis relies on the following key lemma. We let $\{x_t\}_{t\geq 0}$ be the iterates generated by Algorithm~\ref{alg:stoc_prox_linear} and let $\{\xi_t\}_{t\geq 0}$ be the corresponding samples used. For each index $t\geq 0$, we continue to use the notation $\hat x_t=\prox_{\lambda F}^\Phi (x)$ and let 
 $\mathbb{E}_{t}[\cdot]$ to denote the expectation conditioned on all the realizations $\xi_0,\xi_1,\ldots, \xi_{t-1}$.

\begin{lemma} \label{lem:3pt2}
For each iteration $t\geq 0$, the iterates of Algorithm~\ref{alg:stoc_prox_linear} satisfy
\[  \Et{ D_\Phi (\hat x_t, x_{t+1}) }
 \le \frac{1 + \eta_t \tau - \eta_t/\lambda}{ (1-\eta_t\rho) }\cdot 
  D_\Phi (\hat x_t, x_t)  +\frac{1}{4}\cdot \frac{(\sigma \eta_t)^2}{(1-  \eta_t(M+ \tfrac{1}{\lambda}))(1-\eta_t\rho)}.\]
\end{lemma}

\begin{proof}
Following the initial steps of the proof of \Cref{lem:3pt}, we arrive at the estimate \eqref{eq:init_bound}, namely 
\begin{align} \label{eq:init_bound_2}
 &\frac{1}{\eta_t} \Et{ (1-\eta_t\rho) D_\Phi (\hat x_t, x_{t+1}) - D_\Phi (\hat x_t, x_{t}) + D_\Phi (x_{t+1}, x_{t}) }\nonumber\\
 &\hspace{20pt}\le
 \Et{ f_{x_t}(\hat x_t, \xi_t) + r(\hat x_t)-  f_{x_t}(x_{t+1}, \xi_t) - r(x_{t+1}) } .
 \end{align}
  We now seek to bound the right-hand side of \eqref{eq:init_bound_2} using \ref{assum:smooth1}-\ref{assum:smooth2}.
To that end, the following bound  will be useful: 
\begin{align*}
 f_{x_t}( x_{t+1}, \xi_t) 
 &= f(x_t, \xi_t) + \dotp{G(x_t, \xi_t), x_{t+1} - x_t} \\
&\geq  f(x_t, \xi_t) +  \dotp{\nabla f(x_t), x_{t+1} - x_t} - \|G(x_t, \xi_t) - \nabla f(x_t)\|_* \|x_{t+1} - x_t\|.
\end{align*}

\noindent
Taking expectations of both sides and applying Cauchy-Schwarz and \ref{assum:smooth1}-\ref{assum:smooth2}, we obtain
\begin{align}
 \Et{f_{x_t}( x_{t+1}, \xi_t)}
 &\ge  \EE_t\left[f(x_t) +  \dotp{\nabla f(x_t), x_{t+1} - x_t}\right] - \EE_t\left[\|G(x_t, \xi_t) - \nabla f(x_t)\|_* \|x_{t+1} - x_t\|\right]  \nonumber
 \nonumber \\
&\geq \EE_t\left[f(x_{t+1}) - M D_\Phi (x_{t+1}, x_t) \right] - \sqrt{\EE_t\left[\|G(x_t, \xi_t) - \nabla f(x_t)\|_*^2\right]} \sqrt{\EE_t\left[\|x_{t+1} - x_t\|^2\right]} 
\nonumber \\
&\geq \EE_t\left[f(x_{t+1}) - M D_\Phi (x_{t+1}, x_t) \right] - \sigma\sqrt{\EE_t\left[\tfrac{1}{2}\|x_{t+1} - x_t\|^2\right]} \nonumber\\
&\geq \EE_t\left[f(x_{t+1}) - M D_\Phi (x_{t+1}, x_t) \right] - \sigma\sqrt{\EE_t\left[ D_\Phi (x_{t+1},  x_t )\right]} .
\label{eq:rhs}
\end{align}

\noindent
Continuing, add $ f_{x_t} (\hat x_t , \xi_t)$ to both sides of \eqref{eq:rhs}, rearrange, and apply \ref{assum:smooth1} to obtain
\begin{align}
 &\Et{  f_{x_t} (\hat x_t , \xi_t) - f_{x_t}( x_{t+1}, \xi_t)  }\nonumber\\
 &\hspace{20pt}\leq \Et{  f_{x_t} (\hat x_t , \xi_t) - f( x_{t+1}) +  MD_\Phi(x_{t+1}, x_t) }  +  \sigma\sqrt{\EE_t\left[ D_\Phi (x_{t+1},  x_t )\right]}  \nonumber \\
&\hspace{20pt}\leq \Et{  f (\hat x_t) - f( x_{t+1}) +  \tau D_\Phi(\hat x_t, x_t) + MD_\Phi(x_{t+1}, x_t)  }  +  \sigma\sqrt{\EE_t\left[ D_\Phi (x_{t+1},  x_t )\right]}. \label{eq:a2_and_a4}
\end{align}
On the other hand, by the definition of $\hat x_t$ we have
\[ f(\hat x_t) + r(\hat x_t) + \frac{1}{\lambda} D_\Phi(\hat x_t, x_t) \le f(x_{t+1}) + r(x_{t+1})+ \frac{1}{\lambda} D_\Phi( x_{t+1}, x_t). \]
Inserting this equation into \eqref{eq:a2_and_a4}, we obtain
\begin{align}
&\Et{  f(\hat x_t) + r(\hat x_t) - f( x_{t+1})  - r(x_{t+1}) } \nonumber\\
&\leq \Et{    \left( M + \tfrac{1}{\lambda}\right)  D_\Phi( x_{t+1}, x_t) + \left( \tau - \tfrac{1}{\lambda} \right) D_\Phi(\hat x_t, x_t)  } +\sigma\sqrt{\EE_t\left[ D_\Phi (x_{t+1},  x_t )\right]}. \label{eq:a2_and_a4_new}
\end{align}

\noindent
Combining \eqref{eq:a2_and_a4_new} with \eqref{eq:init_bound_2} gives the estimate
\begin{align*}
& \frac{1}{\eta_t} \Et{  (1-\eta_t\rho) D_\Phi (\hat x_t, x_{t+1}) - D_\Phi (\hat x_t, x_{t}) + D_\Phi (x_{t+1}, x_{t}) }\\
 &\hspace{20pt}\le \left(M+ \tfrac{1}{\lambda} \right)\Et { D_\Phi( x_{t+1}, x_t)} + \left( \tau - \tfrac{1}{\lambda} \right) D_\Phi (\hat x_t, x_t) + \sigma\sqrt{\EE_t\left[ D_\Phi (x_{t+1},  x_t )\right]},
 \end{align*}
 Multiplying through by $\eta_t$ and rearranging, we obtain
\begin{align*}
 &\Et{  (1-\eta_t\rho) D_\Phi (\hat x_t, x_{t+1}) + \left( 1-  \eta_t \left(M+ \tfrac{1}{\lambda} \right)\right) D_\Phi (x_{t+1}, x_{t}) }\\
 &\hspace{20pt}\le
\left( 1 + \eta_t(\tau - \tfrac{1}{\lambda}) \right) D_\Phi (\hat x_t, x_t)  +  \sigma\eta_t\sqrt{\EE_t\left[ D_\Phi (x_{t+1},  x_t )\right]}.
\end{align*}
Now define $\gamma := \sqrt{\EE_t\left[ D_\Phi (x_{t+1},  x_t )\right]}$, and rewrite the above as
 \[ 
 \Et{  (1-\eta_t\rho) D_\Phi (\hat x_t, x_{t+1}) }
 \le
 \left( 1 + \eta_t \tau - \tfrac{\eta_t}{\lambda} \right) D_\Phi (\hat x_t, x_t)  + \sigma \eta_t \gamma - \left(1-  \eta_t(M+ \tfrac{1}{\lambda})\right) \gamma^2.
 \]
 Maximizing the right hand side in $\gamma$, i.e. taking $\gamma = \tfrac{\sigma \eta_t}{2\left(1-  \eta_t\left(M+ \tfrac{1}{\lambda}\right)\right)}$, we conclude
  \[ 
 \Et{  (1-\eta_t\rho) D_\Phi (\hat x_t, x_{t+1}) }
 \le
 \left( 1 + \eta_t \tau - \tfrac{\eta_t}{\lambda} \right) D_\Phi (\hat x_t, x_t)  +\frac{1}{4}\cdot \frac{(\sigma \eta_t)^2}{1-  \eta_t(M+ \tfrac{1}{\lambda})},
 \]
as desired.
\end{proof}

With \Cref{lem:3pt2} at hand, we can now establish a convergence rate of Algorithm~\ref{alg:stoc_prox_linear}.

\begin{theorem} \label{thm:convergence2}
The point $x_{t^*}$ returned by Algorithm~\ref{alg:stoc_prox_linear} satisfies:
\[ 
\E{ D_{\Phi}\left(\prox_{\lambda F}^\Phi (x_{t^*}),x_{t^*}\right) }
\le 
\tfrac{\lambda}{(1-(\tau+\rho)\lambda)} \left(
\frac{\lambda(F^\Phi_\lambda(x_0) - \min F)}{ \sum_{t=0}^T \frac{\eta_t}{1-\eta_t\rho}} +  
\frac{\sigma^2 \sum_{t=0}^T  \frac{ \eta_t^2 }{(1-  \eta_t(M+ 1/\lambda ))(1-\eta_t\rho)} } { 4  \sum_{t=0}^T \frac{\eta_t }{1-\eta_t\rho}}
\right).
\]
\end{theorem}

\begin{proof}
Using \Cref{lem:3pt2}, we obtain
\begin{align*}
\Et{F_\lambda^\Phi (x_{t+1})}
&\le \Et{ f(\hat x_t) + \frac{1}{\lambda} D_\Phi (\hat x_t, x_{t+1}) } \\
&\le \Et { f(\hat x_t) + \frac{1}{\lambda(1-\eta_t\rho)} \left( 
 \left( 1 + \eta_t \tau - \eta_t/\lambda\right) D_\Phi (\hat x_t, x_t)  +\frac{1}{4}\cdot \frac{(\sigma \eta_t)^2}{1-  \eta_t(M+ 1/\lambda )}
\right) } \\
& = F_\lambda^\Phi(x_t) + \frac{\eta_t}{\lambda} \left(\frac{\tau + \rho- 1/\lambda}{1-\eta_t\rho} \right) D_\Phi (\hat x_t, x_t)  +\frac{(\sigma \eta_t)^2}{4\lambda(1-  \eta_t(M+ 1/\lambda ))(1-\eta_t\rho)} 
\end{align*}
Recursing and applying the tower rule for expectations, we obtain
\[ 
\E{F_\lambda^\Phi (x_{T+1})}
\le 
F_\lambda^\Phi(x_0) + \sum_{t=0}^T \left(  \frac{\eta_t}{\lambda} \left(\frac{\tau + \rho- 1/\lambda}{1-\eta_t\rho} \right) \E{D_\Phi (\hat x_t, x_t)}  +\frac{(\sigma \eta_t)^2}{4\lambda(1-  \eta_t(M+ 1/\lambda ))(1-\eta_t\rho)}   \right)
\]
Rearranging and using the fact that $\E{F_\lambda(x_{T+1})} \ge \min F$, we obtain
\[ 
 \sum_{t=0}^T\frac{\eta_t}{\lambda} \left(\frac{1/\lambda - \tau - \rho }{1-\eta_t\rho} \right)\E{D_\Phi (\hat x_t, x_t)}
\le 
F_\lambda^\Phi(x_0) - \min F + \frac{\sigma^2}{4\lambda } \sum_{t=0}^T  \frac{ \eta_t^2 }{(1-  \eta_t(M+ 1/\lambda ))(1-\eta_t\rho)}
\]
or equivalently
\[ 
 \sum_{t=0}^T \frac{\eta_t }{1-\eta_t\rho}\E{D_\Phi (\hat x_t, x_t)}
\le 
\frac{\lambda^2(F_\lambda^\Phi(x_0) - \min F)}{1 - (\tau+\rho)\lambda } + \frac{\lambda \sigma^2}{4(1-(\tau+\rho) \lambda)} \sum_{t=0}^T  \frac{ \eta_t^2 }{(1-  \eta_t(M+ 1/\lambda ))(1-\eta_t\rho)}. 
\]
Dividing through by $ \sum_{t=0}^T \frac{\eta_t }{1-\eta_t\rho}$ and recognizing the left-hand-side as $\EE[D_{\Phi}(\hat x_{t^*},x_{t^*})]$, the result follows.
\end{proof}

\noindent
As an immediate corollary, we obtain a convergence rate for Algorithm~\ref{alg:stoc_prox_linear} with a constant stepsize. 

\begin{corollary} \label{cor:convergence_const_smooth}
For some $\alpha>0$, set $\eta_t=\frac{1}{M+\lambda^{-1}+\alpha^{-1}\sqrt{T+1}}$ for all indices $t=1,\ldots, T$. Then the point $x_{t^*}$ returned by Algorithm~\ref{alg:stoc_prox_linear} satisfies:
\[ 
\E{ D_{\Phi}\left(\prox_{\lambda F}^\Phi (x_{t^*}),x_{t^*}\right) }
\le 
\frac{\lambda^2(F^\Phi_\lambda(x_0) - \min F)+\lambda(\frac{\sigma\alpha}{2})^2}{(1-(\tau+\rho)\lambda)}\cdot\left(\frac{M+\lambda^{-1}-\rho}{T+1}+\frac{1}{\alpha\sqrt{T+1}}\right).
\]
\end{corollary}

\section{Rates in function value for convex problems}\label{sec:convexity}
In this final section, we examine convergence rates for stochastic model based minimization under convexity assumptions and prove rates of converge on function values. To this end, we will use the following definition from \cite{rel_smooth_freund}. A function $g\colon\R^d\to\R\cup\{\infty\}$ is {\em $\mu$-strongly convex relative to} $\Phi$ if the function $g-\mu \Phi$ is convex. Notice that $\mu=0$ corresponds to plain convexity of $g$.

In this section, we make the following assumptions: 
\begin{enumerate}[label=(C\arabic*)]
\item \textbf{(Sampling)} It is possible to generate i.i.d. realizations $\xi_1, \ldots, \xi_T \sim P$\label{assum:sampling_convex}

\item \textbf{(One-sided accuracy)} \label{assum:one_sided_convex}
There is a measurable function $(x,y,\xi) \mapsto f_x(y,\xi)$ defined on $U \times U \times \Omega$ satisfying both 
\[ \Exi{f_x(x, \xi) } = f(x), \qquad \forall x \in U  \cap \dom r\]
and
\begin{equation}\label{eqnupper_bound_model_convex}
 \Exi{f_x(y, \xi) }\le f(y), \qquad \forall x,y\in U  \cap \dom r. 
\end{equation}

\item\label{it:convex_model_convex} \textbf{(Convexity of the models)} The exists some $\mu\geq 0$ such that the functions $f_x(\cdot, \xi) + r(\cdot)$ are $\mu$-strongly convex relative to $\Phi$ for all $x\in U\cap \dom r$ and a.e. $\xi \in \Omega$. 

\item \textbf{(Lipschitz property)} \label{assum:lip_convex}
There exists a square integrable function $L \colon \Omega\to \R_+$ such that for all $x,y\in U\cap\dom r$, the following inequalities holds:
\begin{align}
f_x(x, \xi) - f_x(y, \xi)  &\le  L(\xi)\sqrt{D_{\Phi}(y,x)}, \label{eqn:lip_funcmodel_convex}\\
\sqrt{\EE_{\xi}\left[L(\xi)^2\right]} &\leq \lipsymb \nonumber.
\end{align}
\item\label{it:solvability_convex} \textbf{(Solvability)} The convex problems 
$$\min_y\left\{F(y)+\frac{1}{\lambda}D_{\Phi}(y,x)\right\} \qquad \textrm{and}\qquad~\min_y\left\{f_x(y, \xi)+r(y)+\frac{1}{\lambda}D_{\Phi}(y,x)\right\},$$
admit a minimizer for any $\lambda>0$, any $x\in U$, and a.e. $\xi\in \Omega$. The minimizers vary measurably in $(x,\xi)\in  U\times \Omega$.
\end{enumerate}

Thus  the only difference between assumptions \ref{assum:sampling_convex}-\ref{it:solvability_convex} and  \ref{assum:sampling}-(A5) is that
in expectation the stochastic models $f(\cdot,\xi)$ are global under-estimators  \ref{assum:one_sided_convex} and the functions $f(\cdot,\xi)+r(\cdot)$ are relatively strongly convex, instead of  weakly convex \ref{it:convex_model_convex}. Note that  under assumptions \ref{assum:sampling_convex}-\ref{it:solvability_convex}, the objective function $F$ is $\mu$-strongly convex relative to $\Phi$; the argument is completely analogous to that of Lemma~\ref{lem:weak_conv}.

Henceforth, we let $\{x_t\}_{t\geq 0}$ be the iterates generated by Algorithm~\ref{alg:stoc_prox} (with $\tau=\rho=0$) and let $\{\xi_t\}_{t\geq 0}$ be the corresponding samples used. For each index $t\geq 0$, we continue to use the notation $\hat x_t=\prox_{\lambda F}^\Phi (x)$ and let 
 $\mathbb{E}_{t}[\cdot]$ to denote the expectation conditioned on all the realizations $\xi_0,\xi_1,\ldots, \xi_{t-1}$.
We need the following key lemma, which identifies the Bregman divergence $D_{\Phi}(x^*,x_t)$, between the iterates and an optimal solution, as a useful potential function. Notice that this is in contrast to the nonconvex setting, where it was the envelope $F^{\Phi}_{\lambda}(x_t)$ that served as an appropriate potential function.

\begin{lemma} \label{lem:strong_convexity}
For each iteration $t\geq 0$, the iterates of Algorithm~\ref{alg:stoc_prox} satisfy
\begin{align*}
\Et{ (1+\eta_t\mu)D_\Phi (x^\ast, x_{t+1}) }
 &\le
 D_\Phi (x^\ast, x_t)  +\frac{(\lipsymb\eta_t)^2}{4} + \eta_t\Et{r(x_{t}) - r(x_{t+1})}-  \eta_t( F(x_t) - F(x^\ast)),
\end{align*}
where $x^*$ is any minimizer of $F$.
\end{lemma}

\begin{proof}
Appealing to the three point inequality in \Cref{lem:threepoint} and \ref{it:convex_model_convex}, we deduce that all points $x\in \dom r$ satisfy
\begin{equation} \label{eq:3pt_strong_convex}
 f_{x_t}(x, \xi_t) + r(x) + \frac{1}{\eta_t} D_\Phi (x, x_t) 
 \ge f_{x_t}(x_{t+1}, \xi_t)  + r(x_{t+1}) + \frac{1}{\eta_t} D_\Phi (x_{t+1}, x_t) +  \frac{(1+\eta_t\mu)}{\eta_t} D_\Phi (x, x_{t+1}) .
\end{equation}
Setting $x = x^\ast$, rearranging terms, and taking expectations, we deduce
\begin{align} \label{eq:init_bound2_strong_convex}
 &\frac{1}{\eta_t} \Et{ (1+\eta_t\mu) D_\Phi (x^\ast, x_{t+1}) - D_\Phi (x^\ast, x_{t}) + D_\Phi (x_{t+1}, x_{t}) }\nonumber\\
 &\hspace{20pt}\le
 \Et{ f_{x_t}(x^\ast, \xi_t) + r(x^\ast)-  f_{x_t}(x_{t+1}, \xi_t) - r(x_{t+1}) } .
 \end{align}
We seek to upper bound the right-hand-side of \eqref{eq:init_bound2_strong_convex}. Assumptions \ref{assum:one_sided_convex} and \ref{assum:lip_convex} imply:
\begin{align*}
 \Et{  f_{x_t} (x^\ast , \xi_t) - f_{x_t}( x_{t+1}, \xi_t)  } &\le \Et{  f_{x_t} (x^\ast , \xi_t) - f_{x_t}( x_{t}, \xi_t)  + L(\xi) \sqrt{ D_{\Phi}(x_{t+1},x_t)} } \\
 &= \Et{  f_{x_t} (x^\ast , \xi_t) - f(x^\ast)}   + \Et { L(\xi)\sqrt{ D_{\Phi}(x_{t+1},x_t)} } - f(x_t) + f(x^\ast) \\
 &\le\Et { L(\xi)\sqrt{ D_{\Phi}(x_{t+1},x_t)} } - f(x_t) + f(x^\ast).
\end{align*}
The left hand side of \eqref{eq:init_bound2_strong_convex} is therefore upper bounded by
\begin{align*}
& \Et {  L(\xi) \sqrt{ D_{\Phi}(x_{t+1},x_t)}  - f(x_t) - r(x_{t+1} )}+ f(x^\ast) + r(x^\ast) \\
 &=  \Et {  L(\xi) \sqrt{ D_{\Phi}(x_{t+1},x_t)} + (r(x_{t}) - r(x_{t+1})) } - ( F(x_t) - F(x^\ast)).
\end{align*}
Putting everything together, we arrive at
\begin{align*}
&\frac{1}{\eta_t} \Et{ (1+\eta_t\mu)D_\Phi (x^\ast, x_{t+1}) - D_\Phi (x^\ast, x_{t}) + D_\Phi (x_{t+1}, x_{t}) }\\
 &\hspace{20pt}\le
 \Et {  L(\xi) \sqrt{ D_{\Phi}(x_{t+1},x_t)}+(r(x_{t}) - r(x_{t+1})) } -( F(x_t) - F(x^\ast))
 \end{align*}
  Multiplying through by $\eta_t$ and rearranging yields
 \begin{align*}
  \Et{ (1+\eta_t\mu)D_\Phi (x^\ast, x_{t+1}) } &\le
 D_\Phi (x^\ast, x_t)  +   \Et { \eta_t L(\xi) \sqrt{ D_{\Phi}(x_{t+1},x_t)} -{ D_\Phi (x_{t+1}, x_{t}) }}\\
&\hspace{40pt} + \eta_t\Et{r(x_{t}) - r(x_{t+1})} - \eta_t( F(x_t) - F(x^\ast)).
 \end{align*}
Now define $\gamma := \sqrt{\Et {   D_{\Phi}(x_{t+1},x_t)}}$. 
By Cauchy-Schwarz, we have that $\Et { \eta_t L(\xi) \sqrt{ D_{\Phi}(x_{t+1},x_t)}} \leq \eta_t \lipsymb\gamma $. Thus we obtain
\begin{align*}
 \Et{ (1+\eta_t\mu)D_\Phi (x^\ast, x_{t+1}) }
 &\le
 D_\Phi (x^\ast, x_t)  +   \eta_t\lipsymb \gamma -  \gamma^2 +  \eta_t\Et{r(x_{t}) - r(x_{t+1})} - \eta_t( F(x_t) - F(x^\ast)) \\
&\le
 D_\Phi (x^\ast, x_t)  +\frac{(\lipsymb\eta_t)^2}{4} + \eta_t\Et{r(x_{t}) - r(x_{t+1})}-  \eta_t( F(x_t) - F(x^\ast)),
\end{align*}
where the last inequality follows by maximizing the right-hand-side in $\gamma$.
\end{proof}

\noindent

We are now ready to prove convergence guarantees in the case that $\mu = 0$.

\begin{theorem}[Convergence rate under convexity] \label{thm:convergenceC_nonstrong}
For all $T > 0$, we have
\begin{align*}
\EE\left[F\left(\tfrac{1}{\sum_{t=0}^T \eta_t }\sum_{t=0}^T\eta_t x_t\right) - F(x^\ast)\right] \leq \frac{D_\Phi(x^\ast, x_0) + \sum_{t=0}^t \frac{(\eta_t \lipsymb)^2}{4} + \eta_0(r(x_0) - \inf r) }{\sum_{t=0}^T \eta_t},
\end{align*}
where $x^*$ is any minimizer of $F$.
\end{theorem}
\begin{proof}
Lower-bounding the left-hand-side of \Cref{lem:strong_convexity} by zero $D_\Phi (x^\ast, x_{t+1})$, we deduce
\begin{align*}
\eta_t\left[F(x_t) - F(x^\ast) \right] \leq\frac{(\lipsymb\eta_t)^2}{4} + 
 \eta_t\EE_t\left[r(x_{t}) - r(x_{t+1})\right] + \EE_t[D_\Phi (x^\ast, x_t) -  D_\Phi (x^\ast, x_{t+1}) ]
\end{align*}
Applying the tower rule for expectations yields
\begin{align*}
&\sum_{t=0}^T\eta_t\EE\left[F(x_t) - F(x^\ast)\right] \\
&~~\leq \sum_{t=0}^T \frac{(\eta_t \lipsymb)^2}{4} + \EE\left[\sum_{t=0}^T\eta_t(r(x_t) -r(x_{t+1}))\right]+\EE\left[\sum_{t=0}^T( D_\Phi (x^\ast, x_t) -  D_\Phi (x^\ast, x_{t+1}))\right].
\end{align*}
Using Jensen's inequality, telescoping and using the auxiliary Lemma~\ref{lem:aux-lem},
we conclude
\begin{align*}
\EE\left[F\left(\frac{1}{\sum_{t=0}^T}\sum_{t=0}^T \eta_tx_t\right) - F(x^\ast)\right] &\leq \frac{D_\Phi(x^\ast, x_0) + \sum_{t=0}^t \frac{(\eta_t \lipsymb)^2}{4} + \eta_0(r(x_0) - \inf r) }{\sum_{t=0}^T \eta_t},
\end{align*}
as claimed.
\end{proof}

As an immediate corollary of \Cref{thm:convergence}, we have the following rate of convergence when the stepsize $\eta_t$ is constant. 
\begin{corollary}[Convergence rate under convexity for constant stepsize] For any $\alpha>0$ and corresponding constant stepsize $\eta_t = \frac{\alpha}{\sqrt{T+1}}$, we have
\begin{align*}
\EE\left[F\left(\frac{1}{T+1}\sum_{t=0}^T x_t\right) - F(x^\ast)\right] \leq \frac{D_{\Phi}(x^\ast, x_0) + \frac{(\alpha\lipsymb )^2}{4} + \alpha( r(x_0) - \inf r)}{\alpha\sqrt{T+1}},
\end{align*}
where $x^*$ is any minimizer of $F$.
\end{corollary}

\bigskip 

The final result of this section proves that Algorithm~\ref{alg:stoc_prox}, with an appropriate choice of stepsize, drives the expected error in function values to zero at the rate $\widetilde O(\frac{1}{k})$, whenever $\mu > 0$.

\begin{theorem}[Convergence rate strongly convex case] \label{thm:convergenceC}
Suppose that $\eta_t = \tfrac{1}{\mu(t+1)}$ for all $t \geq 0$. Then for all $T > 0$, we have
\begin{align*}
\EE\left[F\left(\frac{1}{T+1}\sum_{t=0}^T x_t\right) - F(x^\ast) + \mu D_\Phi (x^\ast, x_{T+1})\right]  \leq \frac{ \frac{\lipsymb^2(1 + \log(T+1))}{4\mu} + r(x_0) - \inf r + \mu D_{\Phi}(x^\ast, x_0)}{T+1}
\end{align*}
where $x^*$ is any minimizer of $F$.
\end{theorem}

\begin{proof}
Using \Cref{lem:strong_convexity} and the law of total expectation, we have
\begin{align*}
\EE\left[F(x_t) - F(x^\ast) \right] \leq\frac{\eta_t\lipsymb^2}{4} + 
 \EE\left[(r(x_{t}) - r(x_{t+1})) + \frac{1}{\eta_t}D_\Phi (x^\ast, x_t) -  \frac{(1+\eta_t\mu)}{\eta_t}D_\Phi (x^\ast, x_{t+1}) \right]
\end{align*}
Setting $\eta_t = \tfrac{1}{\mu(t+1)}$, averaging, and applying Jensen's inequality yields 
\begin{align*}
\EE\left[F\left(\frac{1}{T+1}\sum_{t=0}^T x_t\right) - F(x^\ast)\right] &\leq \frac{1}{T+1}\sum_{t=0}^T \frac{\lipsymb^2}{4\mu(t+1)} + 
 \frac{\EE\left[r(x_{0}) - r(x_{T+1}) \right]}{T+1} \\
 &\hspace{20pt}+ \frac{1}{T+1}\sum_{t=0}^T\EE\left[\mu(t+1)D_\Phi (x^\ast, x_t) -  \mu(t+2)D_\Phi (x^\ast, x_{t+1}) \right] \\
 &\leq \frac{ \frac{\lipsymb^2(1 + \log(T+1))}{4\mu} + r(x_0) - \inf r + \mu D_{\Phi}(x^\ast, x_0)}{T+1} \\
 &\hspace{20pt} - \EE\left[\frac{(T+2)}{(T+1)} \mu D_\Phi (x^\ast, x_{T+1})\right],
\end{align*}
where the last inequality follows from telescoping the terms in the sum and 
using the lower bound $r(x_{T+1}) \geq \inf r$. This completes the proof. 
\end{proof}

\appendix
\section{Proofs of auxilliary results}
\subsection{Proof of Proposition~\ref{prop:bregpoly}}\label{appendix:bregpoly}
Let us write 
$$\Phi=\widehat \Phi+\widetilde{\Phi},$$
for the two functions
$$
\widehat \Phi(x) := \sum_{i=0}^n \frac{a_i}{i+2}\|x\|^{i+2}_2\qquad \textrm{and} \qquad \widetilde \Phi(x) := \sum_{i=0}^n 3a_i\|x\|^{i+2}_2.
$$
The result \cite[Equation~(25)]{Lu_mirror_weird} yields the estimate
$$
D_{\widehat \Phi}(y, x) \geq \frac{1}{2}\sum_{i=0}^n a_i\|x\|_2^i \cdot \|x - y\|_2^2\qquad\qquad \forall x,y.
$$
Thus the proof will be complete once we establish the inequality,
\begin{align}\label{eq:tildedivergence}
D_{\widetilde \Phi}(y, x) \geq \frac{1}{2}\sum_{i=0}^n a_i\|y\|_2^i \cdot \|x - y\|_2^2\qquad\qquad \forall x,y.
\end{align}
To this end, fix an index $i$, and set $\eta := 3(i+2)$ and $\widetilde \Phi_i(x):=3a_i\|x\|^{i+2}_2$. We will  show
$$
D_{\widetilde \Phi_i}(y, x) \geq \frac{a_i}{2} \|y\|^i_2 \cdot  \|x-y\|^2_2,
$$ 
which together with the identity, $D_{\tilde \Phi}(y, x) = \sum_{i=0}^n D_{\widetilde \Phi_i}(y, x) ,$
completes the proof of~\eqref{eq:tildedivergence}.

A quick computation shows that 
$$D_{\widetilde \Phi_i}(y,x)=3a_i\left(\|y\|_2^{i+2} + (i+1) \|x\|_2^{i+2} - (i+2)\|x\|_2^i\dotp{x, y} \right).$$
Let us consider two cases. First suppose that $\eta^{1/i}\|x\|_2 \geq \|y\|_2$. In this case,~\cite[Proposition~5.1]{Lu_mirror_weird} implies
$$
D_{\widetilde\Phi_i}(y, x) \geq \frac{a_i\eta}{2}\|x\|^i_2  \cdot\|x - y\|^2_2 \geq \frac{a_i}{2}\|y\|^i_2 \cdot\|x - y\|^2_2 ,
$$
as desired.

Now suppose that $\|y\|_2 \geq \eta^{1/i}\|x\|_2$. We will show that $D_{\widetilde\Phi_i}(y, x) \geq \eta^{-1}D_{\widetilde \Phi_i}(x,y)$, which will complete the proof since
$$
\eta^{-1}D_{\widetilde \Phi_i}(x,y) \geq \frac{a_i}{2} \|y\|^i \cdot  \|x-y\|^2_2,
$$
by~\cite[Proposition~5.1]{Lu_mirror_weird}. To that end, we compute
\begin{align*}
&D_{\widetilde \Phi_i}(y, x) = 3a_i\left(\|y\|_2^{i+2} + (i+1) \|x\|_2^{i+2} - (i+2)\|x\|_2^i\dotp{x, y} \right)\\
&\geq \eta^{-1}D_{\widetilde\Phi_i}(x,y) = \tfrac{a_i }{i+2}\left(\|x\|_2^{i+2} + (i+1) \|y\|_2^{i+2} - (i+2)\|y\|_2^i\dotp{x, y} \right)\\
\iff & (1-\eta^{-1}(i+1) )\|y\|_2^{i+2} + \eta^{-1}(i+2)\|y\|_2^i\dotp{x, y} \geq (\eta^{-1} - (i+1))\|x\|_2^{i+2} +  (i+2)\|x\|_2^i\dotp{x, y}\\
\impliedby &  (1-\eta^{-1}(i+1) )\|y\|_2^{i}\left(\|y\|^2 + \frac{\eta^{-1}(i+2)}{(1-\eta^{-1}(i+1) )}\dotp{x, y}\right) \geq (i+2)\|x\|_2^{i}\dotp{x, y}.
\end{align*}
Let us show that the last inequality is true: First, we upper bound the right hand side
$$
(i+2)\|x\|_2^{i}\dotp{x, y} \leq \frac{(i+2)}{\eta^{(1+i)/i}} \|y\|^{i+2}.
$$
Next, we lower bound the left hand side:  
\begin{align*}
&(1-\eta^{-1}(i+1) )\|y\|_2^{i}\left(\|y\|^2 + \frac{\eta^{-1}(i+2)}{(1-\eta^{-1}(i+1) )}\dotp{x, y}\right)  \\
&\geq(1-\eta^{-1}(i+1) )\left(1 - \frac{\eta^{-1}(i+2)}{\eta^{1/i}(1-\eta^{-1}(i+1) )}\right) \|y\|_2^{i+2}\\
&=  \left(1 - \eta^{-1} (i + 1 + \tfrac{(i+2)}{\eta^{1/i}})\right) \|y\|_2^{i+2}.
\end{align*}
Therefore, we need only verify that $\eta$ satisfies 
\begin{align*}
 \frac{(i+2)}{\eta^{(1+i)/i}} &\leq  \left(1 - \eta^{-1} (i + 1 + \frac{(i+2)}{\eta^{1/i}})\right) \\
  \iff (i+2) ~~&\leq  \eta^{(1+i)/i} - \eta^{1/i}(i + 1) -(i+2)\\
   \iff 2(i+2) &\leq  \eta^{1/i}\left(\eta - (i + 1))\right),
\end{align*}
which holds by the definition of $\eta$.  Thus the result is proved.

\subsection{An auxiliary lemma on sequences. }
\begin{lemma}\label{lem:aux-lem}
Consider any nonincreasing sequence $\{a_t\}_{t\geq 0}\subset\R_{++}$ and any sequence $\{b_t\}_{t\geq 0}\subset \R$. Then for any index  $T\in \mathbb{N}$, we have
$$\sum_{t=0}^T a_t{(b_{t}- b_{t+1})}\leq a_0 (b_0 - b^*),$$
where we set $b^*=\inf_{t\geq 0} b_t$.
\end{lemma}
\begin{proof}
We successively deduce
\begin{align*}
\sum_{t=0}^T a_t{(b_t - b_{t+1})} & =\sum_{t=0}^T a_t{[(b_{t} -b^*)- (b_{t+1}-b^*)]}\\
&= a_0 (b_0-b^*) - a_T (b_{T+1}-b^*) 
+ \sum_{t=0}^{T-1} \left(a_{t+1} - a_t  \right)(b_{t+1} - b^*)  \\
&\leq a_0 (b_0- b^*),
\end{align*} 
as claimed.
\end{proof}

\subsection{Proofs of Propositions~\ref{prop:prox_linear_accuracy} and \ref{prop:prox_linear_lipschitz}}\label{appendix:prox_linear_accuracy}
\begin{proof}[Proof of Proposition~\ref{prop:prox_linear_accuracy}]
Using the fundamental theorem of calculus and convexity of the function $x\mapsto p(\|x\|_2)$  we compute
\begin{align*}
&\|c(x,\xi)+\nabla c(x,\xi)(y-x) - c(y,\xi)\|_2 \\
&= \left\|\int_{0}^1 \left(\nabla c(x + t(y - x), \xi) - \nabla c(x, \xi)\right)(y-x) \,dt \right\|_2\\
&\leq \int_{0}^1 \left\|\nabla c(x + t(y - x), \xi) - \nabla c(x, \xi)\right\|_{\text{op}}\|y-x\|_2\, dt \\
&\leq L_2(\xi)\|y-x\|^2_2\int_{0}^1  \left(p(\|x + t(y-x)\|_2)+ p(\|x\|_2)\right) t \, dt  \\
&\leq L_2(\xi)\|y-x\|^2_2\int_{0}^1  \left((1-t)p(\|x\|_2)+ t p(\|y\|_2))+ p(\|x\|_2)\right) t \, dt  \\
&\leq \frac{2L_2(\xi)}{3}\|y-x\|^2_2\cdot (p(\|x\|_2)+p(\|y\|_2)).
\end{align*}
Hence, we deduce
\begin{align*}
h\left(c(x,\xi)+\nabla c(x,\xi)(y-x) ,\xi\right) - h\left(c(y,\xi),\xi\right) 
&\leq L_1(\xi)\cdot \|c(x, \xi) + \nabla c(x, \xi)(y-x)-c(y,\xi)\|_2\\
&\leq \tfrac{2}{3}L_1(\xi)L_2(\xi)\|y-x\|^2_2\cdot (p(\|x\|_2)+p(\|y\|_2))\\
&\leq \tfrac{4}{3}L_1(\xi)L_2(\xi)\cdot D_{\Phi}(y,x),
\end{align*}
where the last inequality follows form Proposition~\ref{prop:bregpoly}.
Taking expectations yields the claimed guarantee.
\end{proof}

\begin{proof}[Proof of Proposition~\ref{prop:prox_linear_lipschitz}]
We successively compute 
\begin{align*}
h\left(c(x,\xi),\xi\right) - h\left(c(x,\xi)+\nabla c(x,\xi)(y-x) ,\xi\right) 
&= L_1(\xi)\|\nabla c(x,\xi)(y-x)\|_2 \\
&\leq L_1(\xi)L_3(\xi) \cdot \sqrt{q(\|x\|_2)} \| y-x\|_2\\
&\leq \sqrt{2} L_1(\xi)L_3(\xi) \cdot \sqrt{D_{\Phi}(y,x)},
\end{align*}
where the last line follows from~\cite[Equation~(25)]{Lu_mirror_weird}. The result follows. 
\end{proof}

\subsection{Proof of Theorem~\ref{thm:env_diff}}\label{sec:app_proof_diff_breg}
	First we rewrite $F_\lambda^\Phi$, using the definition of the Bregman divergence, as
	\begin{align*}
	F^\Phi_\lambda (x) 
	&= \inf_{y} \left\{ F(y) + \frac{1}{\lambda}\Phi(y) - \frac{1}{\lambda} \ip{\nabla \Phi(x)}{y} \right\} - \frac{1}{\lambda}\Phi(x) + \frac{1}{\lambda} \ip{\nabla \Phi(x)}{x} 
	\\
	&= - \sup_{y} \left\{  \ip{\frac{1}{\lambda} \nabla \Phi(x)}{y}  -  \left( F + \frac{1}{\lambda}\Phi \right) (y) \right\} - \frac{1}{\lambda}\Phi(x) + \frac{1}{\lambda} \ip{\nabla \Phi(x)}{x} 
	\\
	&= - \left( F + \frac{1}{\lambda}\Phi \right)^{\star} \left( \frac{1}{\lambda} \nabla \Phi(x)\right) - \frac{1}{\lambda}\Phi(x) + \frac{1}{\lambda} \ip{\nabla \Phi(x)}{x}.
	\end{align*}
	Note that $F + \frac{1}{\lambda}\Phi$ is closed and $\left( \frac{1}{\lambda} -  (\rho+\tau)\right)$-strongly convex.
	Thus the conjugate $(F + \frac{1}{\lambda}\Phi)^{\star}$ is differentiable. By the chain and sum rules for differentiation, we have
	\begin{align*}
	\nabla F^\Phi_\lambda (x) 
	&=
	- \frac{1}{\lambda} \nabla^2 \Phi(x) \left[ \nabla \left( F + \frac{1}{\lambda}\Phi \right)^{\star} \right] \left( \frac{1}{\lambda} \nabla \Phi(x)\right) + \frac{1}{\lambda} \nabla^2 \Phi(x)x 
	\\
	&= \frac{1}{\lambda} \nabla^2 \Phi(x) \left( x - \left[ \nabla \left( F + \frac{1}{\lambda}\Phi \right)^{\star} \right] \left( \frac{1}{\lambda} \nabla \Phi(x)\right) \right)
	\end{align*}
	The (sub)gradient of a convex conjugate function is simply the set of maximizers in the supremum defining the conjugate, so that
	\begin{align*}
	\left[ \nabla\left( F + \frac{1}{\lambda}\Phi \right)^{\star} \right] \left( \frac{1}{\lambda} \nabla \Phi(x)\right)
	&= \argmax_y \left\{ \ip{\frac{1}{\lambda} \nabla \Phi(x)}{y}  -  \left( F + \frac{1}{\lambda}\Phi \right) (y)\right\} \\
	&=  \argmin_y \left\{ F(y) + \frac{1}{\lambda}D_\Phi (y,x) \right\} \\
	&= \prox_{\lambda F}^\Phi (x).
	\end{align*}
	Putting everything together, we obtain,
	$ \nabla F^\Phi_\lambda (x) 
	= 
	\frac{1}{\lambda} \nabla^2 \Phi(x) \left( x - \hat x \right)
	$,
	as desired.

			\bibliographystyle{plain}
	\bibliography{bib}
\end{document}